\documentclass[matann2]{svjour1}

\usepackage{amsmath}
\usepackage{amsfonts}
\usepackage{latexsym}
\usepackage{mathrsfs}

\usepackage{times} 
\usepackage{txfonts} 

\usepackage{array}
\usepackage{amscd}
\usepackage{a4}
\usepackage[mathcal]{euscript}
\usepackage{amssymb}
\usepackage{epsf, epic, eepic} 
\usepackage[alphabetic]{amsrefs} 
\usepackage{theorem}

\usepackage{graphicx}

\newtheorem{theoremnine}{Theorem \ref{exterior:sheaf-essential}.}

\newtheorem{theoremthree}{Theorem \ref{theorem:boolean:cover}.}

\newtheorem{theoremfour}{Theorem \ref{theorem:sheafequalscellular:boolean}.}

\newtheorem{theoremseven}{Theorem \ref{section3:theorem:deletion.restriction}.}

\newtheorem{theoremeight}{Theorem \ref{exterior:cellular:essential}.}

\spnewtheorem*{remark*}{Remark}{\it}{\rm}
\spnewtheorem*{remarks*}{Remarks}{\it}{\rm}
\spnewtheorem*{firstproof*}{First proof}{\it}{\rm}
\spnewtheorem*{secondproof*}{Second proof}{\it}{\rm}
\usepackage{tikz}
\usetikzlibrary{arrows.meta}
\usepackage{graphicx}
\input xy 
\xyoption{all} 
\usepackage[utf8]{inputenc}
\usepackage{color,soul,soulutf8}
\DeclareMathAlphabet{\ams}{U}{msb}{m}{n}
\DeclareMathAlphabet{\goth}{U}{euf}{m}{n}
\def\sl{\text{SL}}

\def\id{\text{id}}

\def\codim{\text{codim}\,}
\def\HH{\mathcal H}

\def\aa{\alpha}\def\ss{\sigma}
\def\ve{\varepsilon}

\def\Z{\ams{Z}}
\def\C{\ams{C}}
\def\B{\ams{B}}
\def\N{\ams{N}}
\def\0{\mathbf{0}}
\def\1{\mathbf{1}}
\def\quo{/\kern -.45em\sim}
\def\Langle{\langle\kern -2pt\langle}
\def\Rangle{\rangle\kern -1.9pt\rangle}
\newcommand{\rmod}{\vrule width 0mm height 0 mm depth
  0mm_R\mathbf{Mod}}
\newcommand{\ra}{\rightarrow}
\newcommand{\rk}{rk}

\newcommand{\cork}[1]{|\kern0.75pt{#1}\kern1pt|}

\newcommand{\blob}{\bullet}

\newcommand{\HS}[1]{H_{{#1}}}
\newcommand{\HC}[1]{H^{\kern1pt\textrm{cell}}_{{#1}}}

\newcommand{\chiprime}{\chi^\prime_{(L,F)}(1)}
\newcommand{\chiprimek}[1]{\chi^{({#1})}_{L}(1)}

\newcommand{\chiprimeperpk}[1]{\chi^{({#1})}_{(L,F^\perp)}(1)}

\title{Sheaf homology of hyperplane arrangements, Boolean covers and
  exterior powers} 

\dedication{Dedicated to Mary, Emmanuelle and Adrien}

\author{Brent Everitt and Paul
  Turner}

\institute{
{\sc Brent Everitt:} Department of Mathematics, University of York, York
YO10 5DD, United Kingdom. \email{brent.everitt@york.ac.uk}. 
{\sc Paul Turner:} Section de math\'ematiques,  
Universit\'e de Gen\`eve, 7-9 Rue du Conseil-G\'en\'eral, CH-1205, Geneva, Switzerland.
\email{paul.turner@unige.ch}.
}

\titlerunning{Sheaf homology, Boolean covers and exterior powers}
\authorrunning{Brent Everitt and Paul Turner}

\begin{document}

\maketitle

\begin{abstract}
We compute the sheaf homology of the intersection lattice of a
hyperplane arrangement with coefficients in the graded exterior sheaf
$\Lambda^\blob F$ of the natural sheaf $F$. This builds on the results
of our previous paper \cite{EverittTurner19a} where this homology was
computed for $\Lambda^1F=F$, itself a generalisation of an old result
of Lusztig. The computational machinery we develop in this paper is
quite different though: sheaf homology is lifted to what we call
Boolean covers, where we instead compute homology cellularly. A number of tools
are given for the cellular homology of these Boolean covers,
including a deletion-restriction long exact sequence. 
\end{abstract} 

\maketitle


\section*{Introduction}

The combinatorics of a hyperplane arrangement is encapsulated by its
intersection lattice. The homology of
this lattice, with constant coefficients, was first determined in
\cites{Folkman66,Bjorner82}, with
Quillen \cite{Quillen78} showing that it has the homotopy type of a wedge of
spheres. 
Interest in homology may be revived though by taking coefficients in a more
interesting local system, that is to say, in a \emph{sheaf\/} on the
lattice. The resulting sheaf homology $\HS {*}(L\kern-1pt\setminus \kern-1pt
\0; F)$, 
where $L$ is the intersection lattice of a hyperplane arrangement and
$F$ is some interesting (naturally occuring) sheaf, then becomes
worthy of investigation. 

Intersection lattices of hyperplanes arrangements come equipped with a
canonical sheaf as the elements of the lattice are vector spaces. We
call this the natural sheaf, and in \cite{EverittTurner19a} we showed
that the reduced sheaf homology is trivial in all degrees, except the
top one, whose dimension is related to the $\beta$-invariant of the
arrangement, i.e. the derivative of the
characteristic polynomial of $L$ evaluated at $1$ -- see \cite{MR921071}.
This generalises, to an arbitrary arrangement, an old result of
Lusztig \cite{Lusztig74} where he considers the arrangement of all
hyperplanes in a vector space over a finite field. 
There are various other sheaves that can be
put on an intersection lattice -- see \cite{Orlik-Terao92} -- but
they turn out to be what Yuzvinsky \cite{Yuzvinsky91} calls local
sheaves, and so the homology vanishes for general reasons. The natural
sheaf is not local.

In this paper our principal object of interest is the sheaf homology
of $L$ with coefficients in the graded sheaf $\Lambda^\blob F$, where
$F$ is the natural sheaf and $\Lambda^j F$ is the $j$-th exterior
power of $F$. We concentrate first on the case where the arrangement
is {\em essential}, meaning that the intersection of all the hyperplanes
is trivial. Our result here   is: 

\begin{theoremnine}
\label{thm9}
Let $L$ be the intersection lattice of an essential hyperplane
arrangement in a space $V$. Let
$F$ be the natural sheaf on $L$ and $\Lambda^{j}F$ be the
$j$-th exterior power of $F$.
If $\rk(L)\geq 2$ then 
$\HS i (L\kern-1pt\setminus\kern-1pt\0;\Lambda^{j}F)$ is
trivial unless: \\ 
\hspace*{1cm}-- either
$0<i<\rk(L)-1$ and $i+ j = \rk(L) - 1$,
in which case
$$ 
\dim\HS {i}
(L\kern-1pt\setminus\kern-1pt\0;\Lambda^{j}F) =
\frac{ (-1)^{i+1}}{j!} \chiprimek j
$$
\hspace*{1cm} -- or, $i=0$ and  or $j= \rk(L) -1$, in which case
$$
\dim\HS {0}
(L\kern-1pt\setminus\kern-1pt\0;\Lambda^{j}F) =
 \binom{\rk(L)}{j}
- 
\frac{1}{j!}
\chiprimek j
 $$
\hspace*{1cm} -- or, $i=0$ and $j<\rk(L) -1$, in which case 
$$
\dim\HS {0}
(L\kern-1pt\setminus\kern-1pt\0;\Lambda^{j}F) =
\binom{\rk (L)}{j}
 $$
 where  $\chi^{({j})}_{L}(t)$ is the $j$-th derivative of the
 characteristic polynomial of $L$. 
\end{theoremnine}

The case $j=1$ reproduces the main result of \cite{EverittTurner19a},
and the appearance there of the
$\beta$-invariant of the arrangement
is expanded to the appearance of higher derivatives of the
characteristic polynomial that are
related to the dimensions of the higher exterior powers. The graded
Euler characteristic of this (bi-graded) homology is (see Corollary
\ref{exterior:eulerchar:sheaf}) 
$$
 \chi_q\HS * (L\kern-1pt\setminus \kern-1pt
  \0;\Lambda^\blob F) =
- \chi_L(1+q)+(1+q)^{\dim V}
$$
The homology $\HS * (L\kern-1pt\setminus \kern-1pt
\0;\Lambda^\blob F)$ can thus be interpreted as a categorification
of the characteristic polynomial of the hyperplane arrangement,
although we do not pursue this point of view.  We extend the results
above to non-essential arrangements in Theorem
\ref{exterior:sheaf-non-essential}.  
  
Our main computational tool is given by what we call \emph{Boolean
  covers\/}. These are Boolean lattices that keep track of all the
expressions of elements as joins of atoms. As lattices they are
particularly amenable to having their homology computed
\emph{cellularly\/} -- a philosophy that we adopted in
\cite{EverittTurner15}. We then 
make the connection betwen this cellular 
homology of Boolean covers and the sheaf homology of the lattices
being covered. 

This is a two step process. Writing $\widetilde{L}$ for the Boolean
cover of $L$, a number
of spectral sequence arguments establish: 

\begin{theoremthree}
\label{thm3}
  Let $L$ be a graded atomic lattice with sheaf $F$ and
  let $f:\widetilde{L}\ra  L$ be its Boolean cover. Then
$$
\HS{*}(L\setminus\0;F)
\cong
\HS{*}(\widetilde{L}\setminus\0;F).
$$  
\end{theoremthree}

This result also appears in \cite{Lusztig74}*{\S1.2}. The second step is:

\begin{theoremfour}
\label{thm4}
If $B$ is a Boolean lattice and $F$ is a sheaf on $B$ then
$$
\HS{*}(B\setminus\0;F)
\cong
\HC{*}(B\setminus\0;F).
$$  
\end{theoremfour}

If $L$ is an intersection lattice, then for a hyperplane $a$ the
deletion $L_a$ and restriction $L^a$ are lattices of ``smaller''
arrangements -- see \S\ref{section1:posets}. 
The characteristic polynomial of $L$ satisfies a
deletion-restriction relation in terms of $L_a$ and $L^a$,
and our main technical tool is a lift of
this to the setting of the cellular homology of Boolean covers.  

\begin{theoremseven}
\label{thm7}
Let $L$ be a geometric lattice equipped with a sheaf $F$ and let $f:\widetilde{L}\ra
L$ be its Boolean cover. Then for any atom $a\in L$ there is a long exact sequence
$$
\cdots
\ra
\HC{i}(\widetilde{L^a};F)
\ra
\HC{i}(\widetilde{L_a};F)
\ra
\HC{i}(\widetilde{L};F)
\ra
\HC{i-1}(\widetilde{L^a};F)
\ra
\HC{i-1}(\widetilde{L_a};F)
\ra
\cdots
$$  
\end{theoremseven}

This allows us to prove the analogue of Theorem
\ref{exterior:sheaf-essential} for the cellular homology of Boolean
covers:

\begin{theoremeight}
\label{thm8}
Let $L$ be the intersection lattice of an essential hyperplane
arrangement in a space $V$, let
$F$ be the natural sheaf on $L$ and $\Lambda^{j}F$ be the
$j$-th exterior power of $F$.
If $\rk(L)\geq 2$ and $\widetilde{L}\ra L$ is the Boolean cover of $L$, then
$\HC{i}(\widetilde{L};\Lambda^{j}F)$ is trivial unless $0\leq i<\rk(L)$
and 
$i+j=\rk(L)=\dim V$, in which case:
$$
\dim\HC{i}(\widetilde{L};\Lambda^{j}F)=
\frac{{(-1)^i}}{j!}
\chiprimek j
$$
where  $\chi^{({j})}_{L}(t)$ is the $j$-th derivative of the
characteristic polynomial of $L$.   
\end{theoremeight}

Indeed this is proved first, and Theorem
\ref{exterior:sheaf-essential} is a corollary. It is extended to
non-essential sheaves in Theorem
\ref{exterior:cellular-non-essential}. 

The theorems above, indeed all the results of this paper, hold
  for lattices in a range of generalities. The broadest class -- for
  example in
  Theorem 3 --
  are the graded atomic lattices. The proof of the long exact sequence
  in Theorem 7 requires the restriction $L^a$ to also be
  graded atomic; to ensure this we restrict to the smaller class of
  geometric lattices. Specific computations of homology, such as
  Theorems 8 and 9, are done for the natural sheaf
  on the further restricted class of arrangement lattices. Finally,
  for our cellular calculations we restrict yet further to the Boolean
  lattices, although this is purely for conciseness and convenience --
  an analogous result to Theorem 4 holds for the class of
  cellular posets; see \cite{EverittTurner15}*{Theorem 2}.

Working with the Boolean cover takes us quite close to the perspective
of Dansco and Licata \cite{DanscoLicata15}. Motivated by Khovanov
homology-style constructions, they make a number of decorated
hypercubes (some using exterior powers) which give rise to homologies
which categorify the characteristic polynomial, among other things, of
a hyperplane arrangement. Our cellular homology of the Boolean cover
is very much of this type, but in fact the resulting decorated
hypercube is not one they consider. They initiate some computations of
the homology for their examples and it would be interesting to see
further (or full) computations. The techniques we develop for Boolean
covers may be of some use in this regard. 

The structure of the paper is as follows. In \S 1 we discuss the
basics of lattices, arrangements and sheaves. We recall the necessary
background on hyperplane arrangements and their intersection lattices,
sheaves on lattices, and characteristic polynomials. We also introduce
Boolean covers. In \S 2 we move to homology,
first discussing sheaf homology and its basic properties and
calculating the Euler characteristic in the example of interest. We then
discuss a Leray-Serre type spectral sequence needed to make the
connection between a lattice and its Boolean cover. In \S 3 we
introduce the cellular homology of a Boolean lattice with coefficients
in a sheaf. We show that cellular homology computes sheaf homology and
give a number of technical results about cellular homology, of which
the most important is the deletion-restriction long exact
sequence. \S 4 studies the main example of the homology of an
arrangement lattice with coefficients in the exterior powers of the
natural sheaf. After a brief discussion of graded Euler
characteristics, we state and prove our main results first for
essential arrangements and then in the non-essential case.


\section{Lattices, arrangements and sheaves}

This section summarises the basics of posets, lattices and sheaves.
\S\ref{section1:posets} presents basic poset notions and
terminology along with the examples that preoccupy this paper:
the intersection lattices of hyperplane
arrangements. \S\ref{section1:sheaves} gives basic sheaf notions
and constructions and the principal examples: the natural sheaf of a
hyperplane arrangement and its exterior
powers. \S\ref{section1:characteristicpolynomial} recalls the
characteristic polynomial and finally \S\ref{section1:booleancovers}
introduces a key tool in the computation of sheaf homology: the
Boolean cover of a graded 
atomic lattice. 


\subsection{Posets, lattices and arrangements}
\label{section1:posets}

Let $P=(P,\leq) $ be a finite graded poset with rank function
$\rk:P\ra\Z$ (see \cite{Stanley12}*{Chapter 3} for this and other basic poset
terminology in this section). 
A \emph{minimum\/} is an element $\0\in P$ with
$\0\leq x$ for all $x\in P$ 
and a \emph{maximum\/} is an element $\1\in P$ with $x\leq\1$ for all $x\in
P$. We assume $\rk(\0) = 0$.
The \emph{atoms\/} of $P$ are the elements of rank $1$. 
A
poset map $f:Q\ra P$ is a set map such that $fx\leq fy\in P$ if $x\leq
y\in Q$. 

A subset $K\subset P$ is \emph{upper convex\/} if $x\in K$ and $x\leq
y$ implies $y\in K$. If $x\leq y$, the \emph{interval\/} $[x,y]$ consists of those
$z\in P$ such that $ x\leq z \leq y$; if $x\in P$ the interval
$P_{\geq x}$ consists of those $z\in P$ such that $z\geq x$; one
defines $P_{\leq x}$, $P_{> x}$ and $P_{< x}$ similarly. 

A lattice is a poset equipped with a join $\vee$ and a
meet $\wedge$.  
A finite lattice has a minimum $\0$, equal to the meet of all
its elements, and a maximum $\1$, equal to the join.
A graded lattice is \emph{atomic\/} if
every element can be
expressed (not necessarily uniquely)
as a join of atoms, with the convention that the 
empty join is the minimum $\0$. 
The rank
$\rk(L)$ of a graded lattice $L$
is $\rk(L):=\rk(\1)$. 

If $A$ is a finite set then the \emph{Boolean lattice\/} $B=B(A)$
consists of the subsets of $A$ ordered by
inclusion. The result is a graded atomic lattice with
$\rk(x)=|x\kern0.25mm|$, join $x\vee y=x\cup
y$, meet $x\wedge y=x\cap y$, minimum $\0=\varnothing$, maximum $\1=A$
and atoms $A$. Any element has
a unique expression as a join of atoms.

\begin{figure}
\begin{tikzpicture}
\draw [white] (0,0)--(\textwidth,0); 
%
\node at (2,2)
{\includegraphics[scale=0.325]{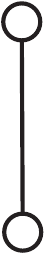}};
\node at (4.5,2)
{\includegraphics[scale=0.325]{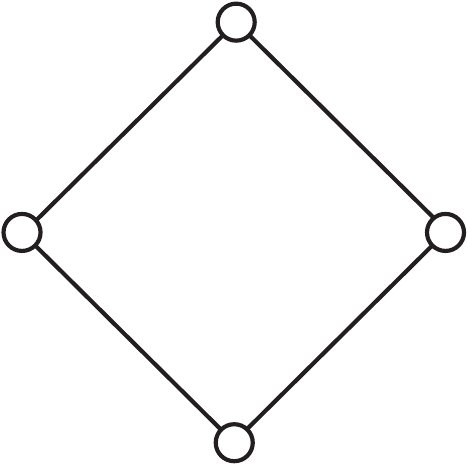}};
\node at (8,2)
{\includegraphics[scale=0.325]{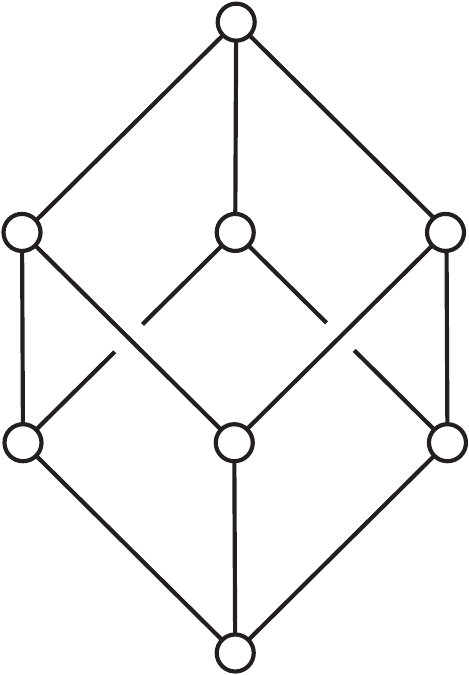}};
\node at (11.5,2)
{\includegraphics[scale=0.325]{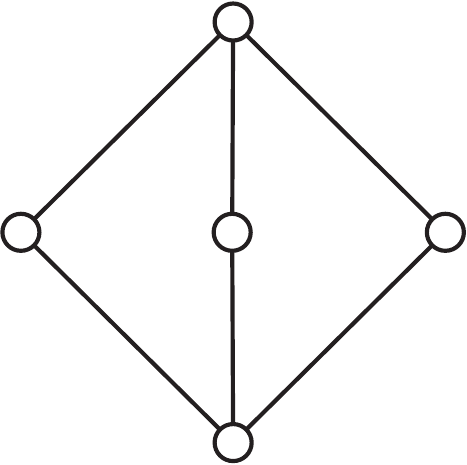}};
%
%
\end{tikzpicture}
\caption{The arrangement lattices $L(A)$ where $|A|\leq 3$.}
\label{fig:arrangement:small}
\end{figure}

This paper is about arrangement lattices. If $V$ is a finite
dimensional vector space over a field $k$, 
then an \emph{arrangement\/} in $V$ is a finite set
$A=\{a_i\}$ of linear hyperplanes, i.e. codimension one subspaces. 
The corresponding \emph{arrangement lattice\/} $L=L(A)$
has elements all possible intersections of hyperplanes
in $A$ -- with the empty intersection taken to be $V$ -- and ordered by
\emph{reverse\/} inclusion. Then $L$ is a graded atomic lattice with atoms
the hyperplanes $A$, rank function $\rk(x)=\codim x$, minimum $\0=V$,
maximum $\1=\bigcap_{a\in A} a$,
$$
x\vee y=x\cap y,\text{ and }x\wedge y=\bigcap z
$$
where the intersection on the right is indexed by the set $\{z\in L:x\cup y\subseteq z\}$. 
Moreover, $L$ is \emph{geometric\/}, in that the rank function
satisfies $\rk(x\vee y)+\rk(x\wedge y)\leq \rk(x)+\rk(y)$. An arrangement is
\emph{essential\/} when 
$\bigcap_{a\in A} a$ is the trivial subspace, or equivalently, $\rk(L)=\dim
V$. The arrangement lattices on at most three hyperplanes are
shown in Figure \ref{fig:arrangement:small}. The first three are
Boolean -- realised by arrangements of coordinate hyperplanes
with respect to a basis in $1,2$ or $3$-dimensions -- and the last is
a braid arrangement (see for instance \cite{MR2383131})
combinatorially isomorphic to the partition lattice $\Pi(3)$ 
of a set of size $3$. 

If $a\in A$ is a hyperplane of an arrangement in $V$, then the
\emph{deletion\/} arrangement in 
$V$ has hyperplanes $A\setminus\{a\}$. Its intersection lattice
$L_a$ consists of the elements of $L$ that can be expressed as a join
of the atoms $A\setminus\{a\}$. The \emph{restriction\/}
arrangement in $a$ has hyperplanes the subspaces $a\cap b$ for $b\in
A\setminus\{a\}$. Its intersection lattice $L^a$ is the interval
$L_{\geq a}=\{x\in L:x\geq a\}$.  

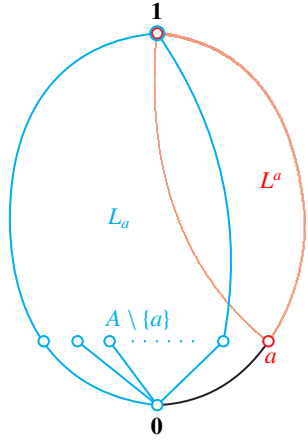
\begin{figure}
\begin{tikzpicture}
\draw [white] (0,0)--(\textwidth,0); 
%
\node at (7,3)
{\includegraphics[scale=0.5]{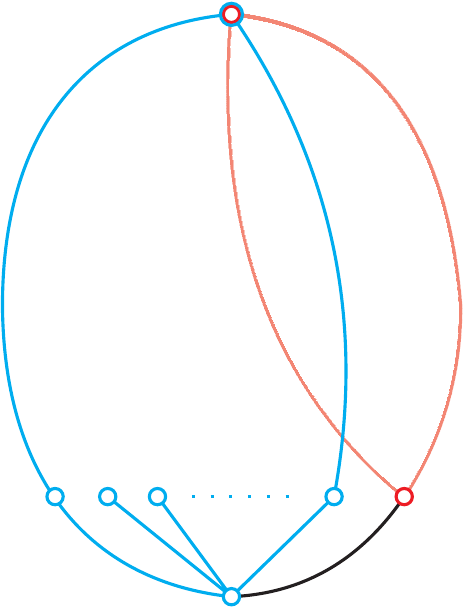}};
\node[color=red]  at(8.5,1.15){$a$};
\node[color=red]  at(8.5,3.5){$L^a$};
\node[color=cyan]  at(6.75,1.65){$A\setminus\{a\}$};
\node[color=cyan]  at(6.5,3){$L_a$};
\node at (7,0.25){$0$};
\node at (7,5.75){$1$};
%
%
\end{tikzpicture}
\caption{The decomposition of $L$ into the deletion $L_a$ and restriction
$L^a$ for a dependent atom $a$.}
\label{fig:arrangement:decomposition}
\end{figure}

In any graded atomic lattice, a set $S\subset A$ of atoms is
\emph{independent\/} if $\bigvee T<\bigvee S$ for all proper subsets
$T$ of $S$, and 
\emph{dependent\/} otherwise. An atom $a$ in a dependent set of atoms $S$ with
the property that $\bigvee S\setminus\{a\}=\bigvee S$ is called a
\emph{dependent atom\/}. A schematic of $L,L_a$ and $L^a$, when $a$ is
dependent, is shown in Figure \ref{fig:arrangement:decomposition}.
It is well known (see for instance
\cites{Birkhoff79,Everitt-Fountain13})
that the only graded atomic lattices without dependent atoms are the
Booleans. Moreover, in a geometric lattice $L$ we have $\rk(\bigvee S)\leq
|S\kern-1pt |$,
and $S$ is
independent if and only if $\rk(\bigvee S)=|S\kern-1pt |$.


\subsection{Sheaves on lattices}
\label{section1:sheaves}

A \emph{sheaf\/} on a poset $P$ is a \emph{contra\/}variant
functor   
$F:P\rightarrow\rmod$
to the category of 
$R$-modules, 
where
$R$ is a commutative ring with $1$, and $P$ is interpreted as a
category in the usual way (having a unique morphism $x\ra y$ whenever $x\leq y$). 
A \emph{morphism\/} of sheaves is a natural
transformation of functors $\kappa:F\rightarrow G$ and an isomorphism
is a natural isomorphism. We write 
$F^y_x$ for the \emph{structure map\/} of the sheaf given by
$F(x\leq y):F(y)\ra F(x)$.

For example, if
$M\in\rmod$ is fixed, then the \emph{constant\/} sheaf $\Delta M$  
has $\Delta M(x)=M$ for every $x\in P$ and 
$(\Delta M)^y_x=id:M\rightarrow M$ for every $x\leq y$ in $P$.

Many sheaf constructions can be done locally, or ``pointwise''. For
example, the direct sum $F\oplus G$ of sheaves $F$ and $G$ has
$(F\oplus G)(x)=F(x)\oplus G(x)$ and structure maps $F_x^y\oplus
G_x^y$ when $x\leq y$. The tensor product $F\otimes G$ can be formed in an analogous way. 
An \emph{($\N$-) graded\/} sheaf $F^\blob$ is a direct sum
$\bigoplus_{i\geq 0}F_i$ of sheaves $F_i$. 

If $Z:\rmod\rightarrow\rmod$ is a functor then we write
$ZF$ for the sheaf arising from the composite
$Z\circ F\colon P\rightarrow\rmod\rightarrow\rmod$. For example, if $F$ is a sheaf
and $j\geq 0$, we have the exterior powers $\Lambda^j F$
of $F$, and hence the graded sheaf:
$$
\Lambda^{\blob} F=\bigoplus_{j\geq 0}\Lambda^jF.
$$
It is easy to check that $\Lambda^{j}\Delta M=\Delta\Lambda^j M$, and that the standard module result:
$$
\Lambda^j(F\oplus G)\cong \bigoplus_{s+t=j}\Lambda^tF\otimes \Lambda^sG
$$
carries straight through to sheaves of modules.


\subsection{The characteristic polynomial}
\label{section1:characteristicpolynomial}

Recall that if $k$ is a field and $L$ is a lattice then the \emph{M\"{o}bius function}
$\mu=\mu_L$ of $L$  is the $k$-valued function on the intervals $[x,y]$ defined recursively by 
$\mu(x,y)=-\sum_{x\leq z<y} \mu(x,z),\text { for all }x<y\text{ in }L$
and $\mu(x,x)=1$. 
If $L$ is an arrangement lattice then the
\emph{characteristic polynomial\/} $\chi_{L}(t)$ is
defined by 
$\chi_{L}(t) = \sum_{x\in L} \mu_L(\0, x) t^{\dim(x)}$.
The $k$-th derivative of $\chi_L$ is denoted $\chi_L^{(k)}$; the value
$(-1)^{\rk(L)-1}\chi^{(1)}(1)$ of the derivative at $1$ is called the
$\beta$-invariant of the arrangement \cite{MR921071}*{7.3}.

We generalise 
to when there is a sheaf $F$
on $L$. The {\em characteristic polynomial of the pair}
$(L,F)$, denoted $\chi_{(L,F)}(t)$,   is defined by
$$
\chi_{(L,F)}(t) = \sum_{x\in L} \mu_L(\0, x) t^{\dim F(x)}.
$$ 
%
If $F$ is the natural sheaf on $L$ then $\chi_{(L,F)}(t) = \chi_{L}(t)$. 


\subsection{Boolean covers}
\label{section1:booleancovers}

Let $L$ be a graded atomic lattice with atoms $A$ and let $B=B(A)$ be the
Boolean lattice on $A$. There is a canonical lattice map $f:B\ra L$ given by 
$$
\textstyle{f:\bigvee_{\kern-0.5mm B}
a_i\mapsto \bigvee_{\kern-0.5mm L}  a_i}
$$
and we refer to the pair $(B,f)$ as 
the \emph{Boolean cover\/} of $L$. We usually write $\widetilde{L}$,
instead of $B$, for the Boolean cover of $L$.
If $F$ is a sheaf on $L$, then there is an induced sheaf $\widetilde{F}$ on
the Boolean cover defined at $x\in \widetilde{L}$ by $\widetilde{F}(x) = F(fx)$
and with structure maps $ \widetilde{F}^y_x = F^{fy}_{fx}\colon
F(fy)\ra F(fx)$. To simplify the notation we will drop the tilde,
writing\footnote{In \cite{EverittTurner15} we wrote $f^*F$ for this
  induced sheaf.} $F$ for $\widetilde{F}$.  
 

For a Boolean lattice $B$ we have $\mu_B(\0, x) = (-1)^{\rk(x)}$;
see \cite{Stanley12}*{Example 3.8.3}. 
Thus,
the characteristic polynomial for $(B,F)$
is given by
$$
\chi_{(B,F)}(t) = \sum_{x\in B} (-1)^{\rk(x)} t^{\dim F(x)}
$$

\begin{proposition} If $\widetilde{L}$ is the Boolean cover of $L$ then
$\chi_{(\widetilde{L},F)}(t) = \chi_{(L,F)}(t)$.
\end{proposition}

\begin{proof}
Unpacking \cite{Orlik-Terao92}*{Lemma 2.35} gives
$\mu_L(\0,x) = \sum_{y\in f^{-1}(x)} (-1)^{\rk(y)}.$
Hence
$$
\chi_{(\widetilde{L},F)}(t)  = \sum_{y\in \widetilde{L}}
                                  (-1)^{\rk(y)} t^{\dim F(y)}
 = \sum_{x\in L} \sum_{y\in f^{-1}(x)}  (-1)^{\rk(y)} t^{\dim F(y)}
= \sum_{x\in L}\mu_L(\0,x)  t^{\dim F(x)}
= \chi_{(L,F)}(t) 
\,\,\,\,\,\qed
$$
\end{proof}


\section{Homology}

In \S\ref{section2:sheafhomology} we recall the basics of the homology of posets with
coefficients in a sheaf and in \S\ref{section2:eulercharacteristics}
we discuss the (graded) Euler characteristic of the resulting
homology. \S\ref{section2:spectralsequences} gives some spectral
sequences that will prove useful in the next section where we compare
(sheaf) homology with the cellular homology defined in
\S\ref{section2:subsection4}.


\subsection{Sheaf homology}
\label{section2:sheafhomology}

For a fixed poset $P$ let $\varinjlim^P\kern-1pt$ be the colimit
functor from sheaves on $P$ to $\rmod$,
and let
$$
\textstyle{\varinjlim_{{\vrule width 0mm height 2.25 mm depth
      0mm}^{*}}^P}
:=L_* \varinjlim^P
$$
be the left derived functors, or \emph{higher colimits\/}. The \emph{homology
$\HS{*}(P;F)$ of $P$ with coefficients in the sheaf $F$\/} are these higher
colimits evaluated at the sheaf $F$. The homology can be computed
using a chain complex 
$S_{\kern-1pt *}(P;F)$ 
whose group of $n$-chains is 
$$
S_{\kern-1pt n}(P;F)=\bigoplus_\ss F(x_0)
$$
where the direct sum is over the totally ordered chains $\sigma =x_n\leq\cdots
\leq x_0$ in $P$. 
For such a chain $\sigma$ 
and $s\in F(x_0)$ we write $s_\ss$ for the element
of $S_{\kern-1pt n}$ that has value $s$ in the component indexed by
$\ss$ and value $0$ in all other components. 
The differential 
$d:S_{\kern-1pt n}(P;F)\ra S_{\kern-1pt n-1}(P;F)$ is then given by
\begin{equation}
  \label{eq:61}
ds_\ss=
F_{x_1}^{x_0}(s)_{d_0\ss} 
+\sum_{i=1}^n (-1)^i s_{d_i\ss} . 
\end{equation}
where as usual
$d_i\sigma  = x_n\leq\cdots \leq \widehat{x}_i \leq \cdots \leq x_0$
for $0\leq i\leq n$. 

We have (see \cite{Gabriel_Zisman67}*{Appendix II})
$$
\HS {*}(P;F) =
\textstyle{\varinjlim_{{\vrule width 0mm height 2.25 mm depth
      0mm}^{*}}^P} F
\cong HS_{\kern-1pt *}(P;F).
$$

The following are some well-known properties of homology.

\begin{lemma}
\label{lemma:basicHSproperties}
  \begin{enumerate}
  \item If $\Delta M$ is a constant sheaf then  $\HS {*}(P;\Delta M)
    \cong H_*(|P\kern1pt|,M)$, the ordinary simplicial homology of the
    order complex $|P\kern1pt|$, which is the geometrical realisation of the
    simplicial complex  
    whose vertices are elements of $P$ and $n$-simplicies
    are chains
    $\sigma =x_n\leq\cdots\leq x_0$.
  \item If $P$ has a minimum or maximum, and $\Delta M$ is a
    constant sheaf, then $\HS {0}(P;\Delta M)=M$ and
    $\HS {i}(P;\Delta M)$ vanishes for $i>0$. 
  \item If $P$ has a minimum $\0$, and $F$ is any sheaf, then
    $\HS {0}(P;F)=F(\0)$ and $\HS {i}(P;F)$ vanishes
    for $i>0$.
  \end{enumerate}
\end{lemma}

Let $T_*(P;F)$ be the chain complex whose $n$-chains are
$T_n(P;F)=\bigoplus_{\ss}F(x_0)$,
the sum is over the 
non-degenerate chains $\ss=x_n<\cdots <x_0$, and 
with differential given by the formula
(\ref{eq:61}). Then $T^*(P;F)$ is a
sub-complex of $S_{\kern-1pt *}$ homotopy equivalent to it (see for
example \cite{EverittTurner19a}*{2.1}).
We will
interchange between the $S_*$ and $T_*$ complexes as convenience dictates.


If $F^\blob$ is a graded sheaf then $\HS {*}(P;F^\blob) = \bigoplus_j \HS {*}(P;F^j)$ has the
structure of a bi-graded vector space.


\subsection{Euler characteristics}
\label{section2:eulercharacteristics}

As usual, the \emph{Euler characteristic\/} of homology is defined to be 
$$
\chi
\HS{*}(P;F)=\sum_{n}(-1)^n\dim\HS{n}(P;F).
$$ 
If $V_\blob$ is an graded vector space then its \emph{graded
dimension\/} is $\dim_q V_\blob=\sum_k \dim V_k \, q^k$, and if $F^\blob$ is
a graded sheaf then the \emph{graded Euler characteristic\/} of the
homology $\HS{*}(P;F^\blob)$ is given by
$$
\chi_q \HS{*}(P;F^\blob)
=
\sum_{n}(-1)^n\dim_q\HS{n}(P;F^\blob)
=
\sum_{n,k}(-1)^n\dim\HS{n}(P;F^k)q^k
=
\sum_{k}\chi \HS{*}(P;F^k)q^k.
$$

\begin{proposition}
The Euler characteristic 
${\displaystyle \chi \HS {*}(L\kern-1pt\setminus \kern-1pt \0; F) = - \sum_{x\in
  L\kern-1pt\setminus \kern-1pt \0} \mu_L(\0,x) \dim F(x)}$
\end{proposition}

\begin{proof}
Let $x\in L$ and define $ch_n(x)$ to be the set of (strict) $n$-chains
in $L\kern-1pt\setminus \kern-1pt \0$ 
of the form $ x_n < \cdots < x_{1}< x_0$, where $x_0=x$. If $\sigma$ is
such a chain we write $\ell(\sigma) = n$. Then by
\cite{Stanley12}*{Proposition 3.8.5} we have 
$$
\mu_L(\0, x) =   \sum_{\sigma\in ch_*(x)} (-1)^{\ell(\sigma)-1}.
$$
Recall that the Euler characteristic is the same as the 
alternating sum of the dimensions of the chain groups in
a complex computing the homology, so
$$
\chi \HS {*}(L\kern-1pt\setminus \kern-1pt \0; F) = \chi
T_{*}(L\kern-1pt\setminus \kern-1pt \0; F) = \sum_{n\geq 0} (-1)^n
\dim T_{n}(L\kern-1pt\setminus \kern-1pt \0; F).
$$
The dimension of $ T_{n}(L\kern-1pt\setminus \kern-1pt \0; F)$ can calculated as
$$
\dim T_{n}(L\kern-1pt\setminus \kern-1pt \0; F) = 
\kern-7pt\sum_{x_n < \cdots < x_{1}< x} \kern-7pt\dim F(x) = 
\sum_{x\in L\kern-1pt\setminus \kern-1pt \0} |ch_n(x)| \dim F(x)
$$
giving
$$
\chi \HS {*}(L\kern-1pt\setminus \kern-1pt \0; F)
= \sum_{n\geq 0} 
\sum_{x\in L\kern-1pt\setminus \kern-1pt \0} (-1)^n |ch_n(x)| \dim F(x)
=
 \sum_{x\in L\kern-1pt\setminus \kern-1pt \0} \sum_{n\geq 0}
      (-1)^n |ch_n(x)| \dim F(x).
$$
The value of the M\"obius function $\mu_L(\0, x)$ may be expressed as
$$
\mu_L(\0, x) = - \sum_n (-1)^n |ch_n(x)|
$$
(see, for example, \cite{Stanley12}*{Proposition 3.8.5}), from which we get

$$
\chi \HS {*}(L\kern-1pt\setminus \kern-1pt \0; F) 
 =    \sum_{x\in L\kern-1pt\setminus \kern-1pt \0} \sum_{n\geq 0}
      (-1)^n |ch_n(x)| \dim F(x) 
=- \sum_{x\in L\kern-1pt\setminus \kern-1pt \0} \mu_L(\0,x) \dim F(x).\hspace*{1.05cm}\qed
$$
\end{proof}

\begin{corollary} \label{cor:derchi}
Writing $\chi^\prime_{(L,F)}(t)$ for the derivative of the characteristic polynomial $\chi_{(L,F)}(t)$, we have
$$
\chi \HS {*}(L\kern-1pt\setminus \kern-1pt \0; F) =  
\dim F(\0)
- \chiprime.
$$
\end{corollary}

\begin{proof}
From the definition of the characteristic polynomial we have
$$
\chi^\prime_{(L,F)}(t) 
 = \sum_{x\in L} \mu_L(\0, x) \dim F(x) t^{\dim F(x)-1}
 $$
so that
$$
\chiprime =  \sum_{x\in L} \mu_L(\0, x) \dim F(x).
$$ 
This then gives
$$
\chiprime
= \mu_L(\0, \0) \dim F(\0) + \sum_{x\in
  L\kern-1pt\setminus \kern-1pt \0} \mu_L(\0, x) \dim F(x) =
\dim F(\0) - \chi \HS {*}(L\kern-1pt\setminus \kern-1pt \0; F).
\hspace*{0.75cm}\qed
$$
\end{proof}



\subsection{Some spectral sequences}
\label{section2:spectralsequences}

There is a Leray-Serre style spectral sequence associated to a poset
map. The following is an adaptation of
\cite{Gabriel_Zisman67}*{Appendix II, Theorem 3.6} -- see also
\cite{EverittTurner19a}*{\S2.3}.

Let $f\colon P \ra Q$ be a poset map and let $F$ be a sheaf on $P$.
For each $q\geq 0$ define a sheaf $H_q^{\text{fib}}$ on $Q$ by
$$
H_q^{\text{fib}}(x) = \HS {q}(f^{-1}Q_{\geq x}; F)
$$ 
for $x\in Q$. 
If $x \leq y$ in $Q$ then the structure map
$H_q^{\text{fib}}(y) \ra H_q^{\text{fib}}(x)$ 
is induced by the
inclusion $ Q_{\geq y} \hookrightarrow Q_{\geq x}$.  

\begin{theorem}[Leray-Serre]
\label{theorem:lerayserregrothendieck}
There is a spectral sequence 
$$
E^2_{p,q}=\HS {p}(Q; H_q^{\text{fib}})\Rightarrow \HS {p+q}(P;F)
$$
\end{theorem}



We are interested in a special case of this spectral sequence which we now describe. 
Let $P$ be a poset equipped with sheaf $F$ and let 
$$
P=\bigcup_{\aa\in K}P_{\aa}
$$
be a covering of $P$ by upper convex subposets. 
We define a poset $N$,
the {\em nerve} of the covering, that mimics the simplicial complex
nerve of a covering of a space. If $X$ is a non-empty subset of the
indexing set $K$, let
\begin{equation}
  \label{eq:1}
P_X = \bigcap_{\aa\in X}P_{\aa}
\end{equation}
Then $N$ is the sub-poset of the Boolean lattice $B(K)$ consisting of
those $X$ for which $P_X\not=\varnothing$. 

For each $q\geq 0$, define a sheaf $\HH_q$ on $N$ by
$$
\HH_q(X)=\HS{q}(P_X;F)
$$
and with structure map 
$\HH_q(X\subset Y):\HS{q}(P_Y;F)\ra \HS{q}(P_X;F)$
induced by the inclusion
$P_Y\hookrightarrow P_X$.

\begin{theorem}
\label{theorem:leray}
Given the set-up of the previous paragraph, there is a spectral sequence 
$$
E^2_{p,q}=\HS{p}(N;\HH_q)\Rightarrow \HS{p+q}(P;F).
$$
\end{theorem}

\begin{proof}
Define a map $f:P\ra N$ by $f(x) = \{\aa \in K : x\in P_\aa\}$. Let
$x\leq y$ in $P$ and suppose that $\aa\in f(x)$, hence $x\in
P_\aa$. As $P_\aa$ is upper convex we have $y\in P_\aa$ too, hence
$\aa\in f(y)$. Thus $f(x)\subseteq f(y)$ in $N$, and $f$ is a poset
map. 

We now claim that for $X\in N$, the fiber $f^{-1}N_{\geq X}$ is the
subposet $P_X$ in (\ref{eq:1}). It then follows that the fiber sheaves $H_*^{\text{fib}}$
of $f$ are the $\HH_*$ above, and hence the result after
applying Theorem \ref{theorem:lerayserregrothendieck}. To see the
claim, we have $x\in P_X$ if and only if $x\in P_\aa$ for all $\aa\in
X$; this in turn happens if and only if
$X\subseteq f(x)$, or equivalently, $x\in f^{-1}N_{\geq X}$. 
\qed
\end{proof}

Lusztig \cite{Lusztig74}*{\S 1.2} gives a simplicial
complex version of this result which he
describes as ``well known'',
although the reader might struggle to find a reference. 


\subsection{Passing to the Boolean cover}
\label{section2:Booleancover}

The spectral sequence of a covering from the previous section (Theorem \ref{theorem:leray}) allows
us to pass from a lattice $L$ to its Boolean cover $\widetilde{L}$
when computing homology. The following result can be found in
\cite{Lusztig74}*{\S 1.2}. 

\begin{theorem}
\label{theorem:boolean:cover}
  Let $L$ be a graded atomic lattice with sheaf $F$ and
  let $f:\widetilde{L}\ra  L$ be its Boolean cover. Then
$$
\HS{*}(L\setminus\0;F)
\cong
\HS{*}(\widetilde{L}\setminus\0;F).
$$
\end{theorem}

\begin{proof}
 We cover $L\setminus\0$ and apply the spectral sequence of Theorem
\ref{theorem:leray}. If $A$ is the set of atoms of $L$, then
$L\setminus\0
=
\bigcup_{a\in A}L_{\geq a}$,
is a covering by upper convex sets. If $X\subseteq A$ then
$$
L_X=\bigcap_{a\in X}L_{\geq a}
=L_{b}
\text{ where }
b=\bigvee_{a\in X}a
$$
as $L_{\geq x}\cap L_{\geq y}=L_{\geq x\vee y}$. Thus
$L_X\not=\varnothing$ for all non-empty $X\subseteq A$, and the nerve
poset $N$ is just the Boolean lattice minus its minimum, i.e.
$N=B(A)\setminus\0=\widetilde{L}\setminus\0$. The sheaf $\HH_q$ is given
by
$$
\HH_q(X)
=
\HS{q}(L_X;F)
=\HS{q}(L_{\geq X};F)
=
\left\{
\begin{array}{ll}
F(X),  & q=0,\\
0, & q>0
\end{array}\right.
$$
for $\varnothing\not= X\subseteq A$. Thus $\HH_q$ is the trivial sheaf
when $q>0$ and $\HH_0=F$. The $E^2$-page of the sequence of Theorem 
\ref{theorem:leray} is thus zero except for the $q=0$ line, where 
$E^2_{p,0}=\HS{p}(\widetilde{L}\setminus\0;F)$. This gives the desired
isomorphism 
$
\HS{*}(L\setminus\0;F)
\cong
\HS{*}(\widetilde{L}\setminus\0;F)
$.\qed
\end{proof}


\section{Cellular homology}
\label{section2:subsection4}

The ordinary singular homology of a space 
can be computed cellularly. In
\cite{EverittTurner15} we define a cellular cohomology
that computes, for a large class of posets, the cohomology of a
poset with coefficients in a sheaf. In this section we recall the
basics we need (for homology rather than cohomology), restricting
ourselves to the setting of Boolean lattices, and then reprove a 
theorem of Lusztig relating the homology of a lattice equipped with a
sheaf to the cellular homology of the Boolean cover. 
\S\S\ref{section3:SES}-\ref{section3:zeroes} contain technical results
that give a useful splitting theorem in
\S\ref{section3:splittingbooleans}. This leads to the first
main theorem of the paper, the deletion-restriction long exact
sequence for cellular homology -- Theorem 
\ref{section3:theorem:deletion.restriction} of \S\ref{section3:LES}. 


\subsection{Basics}
\label{section3:basics}

Let $B=B(A)$ be the Boolean lattice on the finite set $A$  and let
$F$ be a sheaf on $B$. Pick an ordering on $A$ and write $A=\{a_1,a_2,\ldots,a_n\}$. An
element $x\in B$ is a subset of $A$, say 
$x=\{a_{i_1},\ldots, a_{i_k}\}$, which we write as $x=a_{i_1}\ldots
a_{i_k}$ assuming that $i_m<i_n$ for $m<n$. If  
$$ 
y=a_{i_1}\ldots\widehat{a_{i_j}}\ldots a_{i_k}
$$
for some $j$ then define $\ve_y^x:=(-1)^{j-1}$. 

The \emph{cellular chain complex\/} $C_*(B;F)$ has $k$-chains 
$$
C_k(B;F):=\bigoplus_{\rk(x)=k} F(x)
$$
where the direct sum is over the subsets $x$ of size $k$. The differential $d:C_k\ra
C_{k-1}$ is given by $d=\sum_{y<x}d_y^x$, where the sum is over the pairs
$y<x$ with $y$ of size $k-1$ and $x$ of size $k$, and with
$d_y^x=\ve_y^xF_y^x$. 

If $z$ is a subset of size $k-2$ with $z< x$ and $y_1,y_2$ are the
two subsets of size $k-1$ with $z< y_1,y_2< x$, then 
\begin{equation}
  \label{eq:5}
\ve_z^{y_1}\ve_{y_1}^x+\ve_z^{y_2}\ve_{y_2}^x=0.  
\end{equation}
It follows that $d^2=0$ and $C_*(B;F)$ is a chain complex. Call
$\HC{*}(B;F)$ the \emph{cellular homology\/} of $B$ with coefficients in
$F$. Up to isomorphism this construction is independent of the order
chosen on $A$ and of the sign assignment used (any collection of
$\ve^x_y$ satisfying (\ref{eq:5}) will do).


\subsection{Sheaf = Cellular}
\label{section3:lusztig}

For a Boolean lattice we now have two kinds of homology -- sheaf
and cellular -- and we now show these are isomorphic. The proof of the
following is adapted from \cite{EverittTurner15}*{Theorem 2}

\begin{theorem}
\label{theorem:sheafequalscellular:boolean}
If $B$ is a Boolean lattice and $F$ is a sheaf on $B$ then
$$
\HS{*}(B\setminus\0;F)
\cong
\HC{*}(B\setminus\0;F).
$$
\end{theorem}

The proof filters the complex $S_{\kern-1pt *}(L;F)$ so that the
standard spectral sequence has $E^1$-page with single non-zero row
the cellular chain complex $C_{\kern-1pt *}(L;F)$.

\begin{proof}
Write $L=B\setminus\0$ and filter the
complex $S_{\kern-1pt *}(L;F)$ by defining $F_pS_{\kern-1pt
  *}=S_*(L_p;F)$, where $L_p=\{x\in 
L:\rk_L(x)\leq p\}$, the elements whose rank \emph{in $L$\/} is at most $p$. The
$E^0$-page of the standard spectral sequence of a filtration is then
$$
E^0_{pq}=\frac{S_{\kern-1pt p+q}(L_p;F)}{S_{\kern-1pt p+q}(L_{p-1};F)}
$$
a quotient complex that we denote by $S_{\kern-1pt *}(L_p,L_{p-1};F)$. 

The
$E^1$-page is $E^1_{pq}=\HS{p+q}(L_p,L_{p-1};F)$. 
Analysing this homology a little further, the arguments of 
\cite{EverittTurner15}*{\S 2}
can be adapted to show 
$$
\HS{*}(L_p,L_{p-1};F)
\cong
\bigoplus_{\rk_L(x)=p}\HS{*}(L_{\leq x},L_{<x};\Delta F(x))
\cong
\bigoplus_{\rk_L(x)=p}\widetilde{H}_{*-1}(L_{<x};\Delta F(x)).
$$
Here, as we have the constant sheaf $\Delta F(x)$, the homology
$\widetilde{H}_*(L_{<x};\Delta F(x))$ is just the ordinary reduced singular
homology of the order complex of $L_{<x}$. 

The poset $L_{<x}$ is isomorphic to a Boolean lattice of rank
$\rk_L(x)$ minus its minimum and maximum. This, in turn, may be identified
with the poset of sub-simplices of the boundary of a standard
$\rk_L(x)$-simplex. Thus, 
the order complex 
$|L_{<x}|$ is a $(\rk(x)-1)$-sphere and
$$
\widetilde{H}_{i-1}(L_{<x};\Delta F(x))
\cong
\left\{
\begin{array}{ll}
  F(x),& i=\rk(x)\\
  0,&\text{ else}.
\end{array}\right.
$$

It follows that $\HS{i}(L_p,L_{p-1};F)=0$ when $i\not=p$, so the
$E^1$-page is trivial except along the $q=0$ line, where 
$$
E^1_{p,0}=\HS{p}(L_p,L_{p-1};F)
=\bigoplus_{\rk_L(x)=p}\kern-2pt F(x)=C_p(L;F)
$$ 
is the module of cellular $p$-chains. The differential 
$
\HS{p-1}(L_{p-1},L_{p-2};F)
\leftarrow
\HS{p}(L_p,L_{p-1};F)
$
coincides with the cellular differential $C_{p-1}\leftarrow C_p$,
and thus
$
\HS{*}(L;F)
\cong
\HC{*}(L;F)
$.\qed
\end{proof}


As a corollary to Theorems \ref{theorem:boolean:cover} and
\ref{theorem:sheafequalscellular:boolean}  
we obtain a result of Lusztig
\cite{Lusztig74}*{Chapter 1}, who proves that the homology of a lattice
with 
coefficients in a sheaf is isomorphic to the \emph{cellular\/}
homology of the Boolean 
cover equipped with the induced sheaf:

\begin{corollary}[Lusztig]
\label{theorem:sheafequalscellular}
Let $L$ be a graded atomic lattice with sheaf $F$ and
let $\widetilde{L}\ra  L$ be its Boolean cover. Then,
$$
\HS{*}(L\setminus\0;F)
\cong
\HC{*}(\widetilde{L}\setminus\0;F),
$$
\end{corollary}


\subsection{Short exact sequences for cellular homology}
\label{section3:SES}

There are two short exact sequences of cellular chain
complexes that will prove useful.

\paragraph{The sequence induced by a sub-Boolean.} Let $B=B(A)$ be
the Boolean lattice on the set $A$ and let $x\in B$. As $x$ is a subset of $A$ we
can consider the Boolean $B(x)$ -- consisting of the subsets of $x$
ordered by inclusion -- and this is naturally a sub-poset of $B$ with
minimum $\0$ and maximum $x$. If $x=A\setminus\{a\}$ then $B(x)$ is
just the deletion $B_a$; if $x=\varnothing$ then $B(x)=\0$. 

If $F$ is a sheaf on $B$, then (up to choice of signage in
constructing the differential) the cellular complex $C_*(B(x);F)$ is a 
subcomplex of $C_*(B;F)$. Moreover, the quotient complex can be easily
described: it is a ``cellular like'' complex of $B\setminus
B(x)$. Specifically, let 
$$
C_k=\bigoplus_{y} F(y)
$$
the direct sum over the subsets $y$ of size $k+1$ such that
$y\not\leq x$. Define $d:C_k\ra C_{k-1}$ as before:
$d=\sum_{w< y}d_w^y$, where $y$ has one more element than $w$, but
where now both $w,y\not\leq x$. If 
$z\not\leq x$ has size $k-1$ and $z\leq y$, then the $y_1,y_2$ of size
$k$ with $z<
y_1,y_2< y$ are also such that $y_1,y_2\not\leq x$. It follows
from (\ref{eq:5}) that $d^2=0$. Write $C_*(B\setminus\kern-1pt B(x);F)$ for the
resulting complex. 

There is then a short exact sequence of cellular complexes
\begin{equation}
  \label{eq:6}
0\ra C_*(B(x);F)\ra C_*(B;F)\ra C_{*-1}(B\setminus\kern-1pt B(x);F)\ra 0  
\end{equation}
If $x=A\setminus\{a\}$ then $B\setminus\kern-1pt B(x)$ is the
restriction $B^a$, which is again a Boolean lattice.

\paragraph{The sequence induced by a short exact sequence of sheaves.}
Let $F$ and $G$ be sheaves on the Boolean $B=B(A)$ and $\kappa = \{\kappa_y\}:G\ra F$
a map 
of sheaves. Then there is an induced map $\kappa_*:C_*(B;G)\ra C_*(B;F)$
defined by 
$$
\kappa_*:s_y\mapsto \kappa_y(s_y)
$$
where $s_y\in C_k(B;G)$ has value $s\in F(y)$ in the coordinate
indexed by the $k$-subset $y$, and value $0$ elsewhere. 
Then
$\kappa_*$ is a chain map and moreover, the cellular chain complex
$C_*(B;\text{--})$ is an exact functor from the category of sheaves on
$B$ to the category
of chain complexes. Thus, a short exact sequence of sheaves
$$
0\ra G\ra F\ra H\ra 0
$$
induces a short exact sequence of cellular complexes
\begin{equation}
  \label{eq:2}
0\ra C_*(B;G)\ra C_*(B;F)\ra C_*(B;H)\ra 0.
\end{equation}

\begin{corollary}\label{cor:directsum}
Let $F$ and $G$ be sheaves on the Boolean lattice $B$. Then,
$$
\HC * (B;F\oplus G )\cong  \HC * (B;F ) \oplus \HC * (B;G )
$$
\end{corollary}


\subsection{Fiddling with $\0$}
\label{section3:zeroes}

We saw in Lemma \ref{lemma:basicHSproperties} that a minimum
needs to be removed for sheaf homology to be meaningful. Corollary
\ref{theorem:sheafequalscellular} above transfers this requirement
to the Boolean cover. Nevertheless, it will turn out to be more
convenient to leave the miniumum \emph{in\/} when performing
calculations with Boolean covers. This section marries the two points of view.

\begin{proposition}
\label{section3:prop:minimums}
Let $B$ be a Boolean lattice and let $F$ be a sheaf on $B$.
Then $\HC i(B\setminus\0;F)\cong \HC{i+1}(B;F)$ for $i>0$,
and in low degrees there is an exact sequence 
$$
0\ra \HC 1(B;F)\ra \HC 0(B\setminus\0;F)\ra F(\0)\ra \HC 0(B;F)\ra 0
$$
\end{proposition}

\begin{proof}
If we take $x=\0$ in the sequence induced by a sub-Boolean 
in \S\ref{section3:SES} we get a short exact sequence 
\begin{equation*}
0\ra C_*(\0;F)\ra C_*(B;F)\ra C_{*-1}(B\setminus\kern-1pt \0;F)\ra 0.  
\end{equation*}
The result follows immediately from the associated long exact sequence.
 \qed
\end{proof}

Putting this together with Corollary \ref{theorem:sheafequalscellular}, we get the
sheaf homology $\HS *(L\setminus\0;F)$ in terms of the cellular homology
$\HC {*}(\widetilde{L};F)$:

\begin{proposition}
\label{section3:fiddlingwith0:sheaf}
If $L$ is a graded, atomic lattice then 
$\HS i (L\setminus\0;F) \cong \HC {i+1} (\widetilde{L};F)$
for $i>0$, 
and
$$
\dim \HS 0 (L\setminus\0;F) = \dim \HC {1} (\widetilde{L};F)   - \dim
\HC {0} (\widetilde{L};F) + \dim F(\0). 
$$ 
\end{proposition}
\begin{proof}
For $i>0$ apply Corollary \ref{theorem:sheafequalscellular} and Proposition
\ref{section3:prop:minimums}. For degree zero, consider the low degree
short exact sequence of Proposition \ref{section3:prop:minimums}: 
$$
0\ra \HC 1(\widetilde{L};F)\ra \HC 0(\widetilde{L}\setminus\0;F)\ra
F(\0)\ra \HC 0(\widetilde{L};F)\ra 0 
$$
Now use Corollary \ref{theorem:sheafequalscellular} to replace $\HC
0(\widetilde{L}\setminus\0;F)$ by $\HS 0(L\setminus\0;F)$
and recall that for an exact sequence the alternating sum of dimensions
is zero. 
\qed
\end{proof}

\begin{corollary}
\label{cor:coverchar}
Let $\chi_{(L,F)}^\prime (t)$ be the derivative of the characteristic
polynomial of the pair $(L,F)$. Then the Euler characteristic of the
cellular homology of the Boolean cover is  
$$
\chi\HC * (\widetilde{L};F)  = \chiprime =  \sum_{x\in L} \mu_L(\0,x) \dim F(x).
$$
\end{corollary}

\begin{proof} From the above and Corollary \ref{cor:derchi} we have
$$
\chi\HC * (\widetilde{L};F) = \dim F(\0) - \chi \HS
{*}(L\kern-1pt\setminus \kern-1pt \0; F)  = 
\chiprime  
$$

\end{proof}

\subsection{Splitting Booleans}
\label{section3:splittingbooleans}

An atom $a$ splits a Boolean $B$ into the deletion $B_a$ and the
resriction $B^a$, both of which are themselves Booleans of rank
$\rk(B)-1$. Proposition \ref{section3:prop:doubles} and Theorem
\ref{section3:theorem:decomposing:booleans}  below describe two
situations where such a splitting can give useful information about
the homology of $B$ itself. 


\paragraph{Doubling:} Let $F$ be a sheaf on $B$ for which there is an atom
$a\in A$ such that for all $x\in B_a$ the structure map 
\begin{equation}
  \label{eq:4}
F_{x}^{x\vee a} : F(x\vee a)   \ra F(x)
\end{equation}
is the identity. The restrictions of $F$ to $B_a$ and
$B^a$ are consequently exactly the same sheaf and so we call $(B;F)$ a
\emph{double\/} (see Figure \ref{fig:cellular:double}). 

\begin{figure}
\begin{tikzpicture}
\draw [white] (0,0)--(\textwidth,0); 
%
\begin{scope}[xshift=0mm,yshift=2.5mm]
\node at (7,2)
{\includegraphics[scale=0.65]{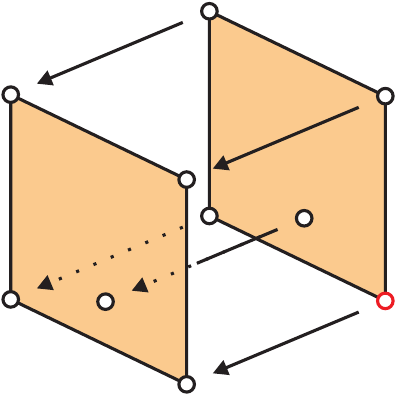}};
\node at (6.45,0.7){$F(x)$};
\node at (8.2,2){$F(x\vee a)$};
\node at (7.2,1.1){$\id$};
\node[color=red] at (9.05,0.6){$F(a)$};
\end{scope}
%
%
\end{tikzpicture}
\caption{A double.}
\label{fig:cellular:double}
\end{figure}
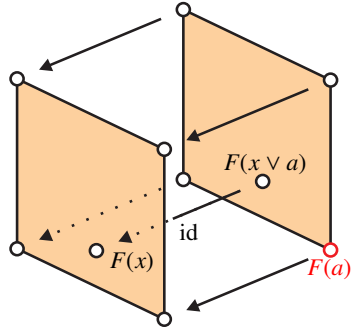

\begin{proposition}
\label{section3:prop:doubles}
Let $(B,F)$ be a double. Then $C_*(B;F)$ is acyclic, i.e.
$\HC{i}(B;F)=0$
for all $i$. 
\end{proposition}

\begin{proof}
Taking $x=A\setminus\{a\}$ in  (\ref{eq:6}), gives the short exact sequence
$$
0\ra C_*(B_a;F)\ra C_*(B;F)\ra C_{*-1}(B^a;F)\ra 0  
$$
from which there results a long exact sequence
$$
\cdots \ra \HC{i}(B^a;F)
\stackrel{\delta}{\ra} 
\HC{i}(B_a;F)\ra \HC{i}(B;F)\ra\HC{i-1}(B^a;F)
\stackrel{\delta}{\ra}  
\HC{i-1}(B_a;F)\ra\cdots
$$
Recall that the signs in the definition of the differential of
cellular homology required a choice of ordering on $A$. By reordering
if necessary we may place $a$ in first
position. It follows that the signs 
$\ve_{x}^{x\vee a}$ are equal to 1 for all $x\in B_a$ and consequently
the connecting homomorphism 
$\delta$ is the map in homology induced by the identity map
$\id: C_*(B^a;F)\ra C_*(B_a;F)$. Thus, $\delta$ is an isomorphism and the result
follows. 
 \qed
\end{proof}

\paragraph{Decomposing:} A small generalisation of the doubling idea gives a very useful
recursive procedure for computing cellular homology.
We will use it for example in \S\ref{section4:exterior} in the computation of 
$\HS{*}(L\setminus\0;\Lambda^\blob F)$. 
Let $F$ be a sheaf for which there is an atom
$a\in A$ such that for all $x\in B_a$ the structure map 
\begin{equation}
  \label{eq:3}
F_{x}^{x\vee a} : F(x\vee a)  \ra F(x)
\end{equation}
is injective. We will call such a sheaf
\emph{decomposable\/}. 

Let $F^a$ denote the restriction of $F$
to $B^a$ and $F_a$  the restriction  to 
$B_a$. 
Since $B_a=B^a$ we may also equip $B_a$ with the sheaf $F^a$: for $x\in B_a$ set 
$F^a(x)=F(x\vee a)$ and for $x\leq y$ define $F^a(x\leq y)=F^{y\vee a}_{x\vee a}$.
The maps $F^a(x)= F(x\vee a) \ra F(x) = F_a(x)$ define a morphism $F^a
\ra F_a$ of sheaves  on $B_a$, which is injective by condition
(\ref{eq:3}), and we will denote the quotient sheaf by $F_a/F^a$.  

\begin{theorem}
\label{section3:theorem:decomposing:booleans}
Let $B$ be Boolean and $F$ a decomposable sheaf on $B$.
Then
$$
\HC{*}(B;F)\cong\HC{*}(B_a;F_a/F^a)
$$
\end{theorem}

\begin{proof}
There is a short exact sequence of sheaves on $B$:
$$
\begin{tikzpicture}
\draw [white] (0,0)--(\textwidth,0); 
%
\node at (3,1.5)
{\includegraphics[scale=0.45]{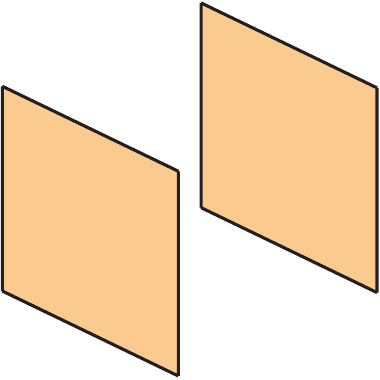}};
\node at (7,1.5)
{\includegraphics[scale=0.45]{fig1a}};
\node at (11,1.5)
{\includegraphics[scale=0.45]{fig1a}};
\node at (2.25,1.25){$F^a$};
\node at (3.75,1.9){$F^a$};
\node at (6.25,1.25){$F_a$};
\node at (7.75,1.9){$F^a$};
\node at (10.25,1.25){$F_a/F^a$};
\node at (11.75,1.9){$0$};
\node at (0.4,1.5){$0$};
\node at (13.6,1.5){$0$};
\node at (3.5,0.5){$G$};
\node at (7.5,0.5){$F$};
\draw[line width=.2mm,-Straight Barb](0.6,1.5)--(1.4,1.5);
\draw[line width=.2mm,-Straight Barb](4.6,1.5)--(5.4,1.5);
\draw[line width=.2mm,-Straight Barb](8.6,1.5)--(9.4,1.5);
\draw[line width=.2mm,-Straight Barb](12.6,1.5)--(13.4,1.5);
%
%
\end{tikzpicture}
$$
The structure maps between the elements of $B^a$ and $B_a$ in the leftmost
sheaf $G$ are all the identity; the middle sheaf is just $F$; the rightmost
sheaf is trivial on $B^a$ and the 
quotient sheaf $F_a/F^a$ on $B_a$. By (\ref{eq:2}) we have a short
exact sequence of cellular chain complexes with resulting long exact sequence:
$$
\cdots\ra\HC{i}(B;G)\ra\HC{i}(B;F)\ra\HC{i}(B_a;F_a/F^a)\ra\cdots
$$
after identifying the cellular homology of the rightmost sheaf with
$\HC{i}(B_a;F_a/F^a)$. But $(B;G)$ is a double (\ref{eq:4}), and so the
result follows from Proposition \ref{section3:prop:doubles}.
 \qed
\end{proof}


\subsection{The deletion-restriction long exact sequence}
\label{section3:LES}

Let $L$ be a geometric lattice and let $f:\widetilde{L}\ra L$ be its
Boolean cover. We will write $B=\widetilde{L}$.
If $a$ is an atom of $L$ then the Boolean cover $\widetilde{L_a}$ of the
deletion $L_a$  can be identified with the sub-Boolean $B_a$ of
$\widetilde{L}$. Under this identification $\widetilde{F_a}$ on
$\widetilde{L_a}$ is the restriction of $F$ (on $B$) to $B_a$. 
Consequently, we just write $F$ for the sheaf
on $\widetilde{L_a}$ induced by a sheaf $F$ on $L$, and
\begin{equation}
\label{eq:Bsuba}
\HC i (B_a;F) = \HC i (\widetilde{L_a};F).
\end{equation}
The Boolean cover $\widetilde{L^a}$ of the restriction $L^a$ is
not, however, the sub-Boolean $B^a$ of $B$: the rank of
the cover is in general less than that of the
sub-Boolean. Nevertheless, they have the same cellular homology.  If
$\widetilde{L^a}\ra L^a$ is the Boolean cover of the 
restriction we also just write $F$ for the sheaf induced on
$\widetilde{L^a}$ by the 
restriction of $F$ to $L^a$.


\begin{theorem}
  \label{section3:theorem:deletion.restriction.auxilary}
Let $L,\widetilde{L}$ and $a\in L$ be as above. Then, for all $i$, 
$$\HC{i}(B^a;F)\cong\HC{i}(\widetilde{L^a};F).$$ 
\end{theorem}

\begin{proof} 
Write $\B=B^a=(\widetilde{L})^a$. 
Let $A$ be the set of atoms of $L$ and let $A_a= A \setminus \{a\}$ be
the atoms of $L_a$.  The set of 
atoms of the restriction $L^a$ is $A^a=\{b\vee_L a:b\in A_a\}$.   
The atoms of the sub-Boolean
$\B$ are the elements $b\vee_B a$ where $b\in A_a$. Note that these are all distinct.

If the elements $b\vee_L a$ for $b\in A_a$ are all 
are distinct (in $L$ itself), then the Boolean cover $\widetilde{L^a}$ is precisely the
sub-Boolean $\B$, and the result
follows. 

Otherwise, there exist distinct atoms $s=a\vee_B b$ and $s'=a\vee_B b'$
of $\B$
that are mapped by $f$ to the same atom $a\vee_L b=a\vee_L b'$ of
$L^a$. As usual let $\B_s$ and $\B^s$ denote the deletion and
restriction of $\B$ with respect to the atom $s$. We claim that, for
all $i$,  
\begin{equation}
  \label{eq:7}
  \HC{i}(\B;F)\cong\HC{i}(\B_s;F)
\end{equation}
To prove this we will show  that $\B^s$ is a double. 
Let $\alpha = s\vee_B s^\prime = a \vee_B b \vee_B b^\prime$. This is
an atom of $\B^s$ and we may consider the deletion 
$(\B^s)_\alpha$ and the restriction $(\B^s)^\alpha$. Note that
$\alpha$ and $s$ are mapped by $f$ to the same element of $L^a$: 
$$
f(\alpha) = f(a \vee_B b \vee_B b^\prime) = a \vee_L b \vee_L b^\prime
= (a \vee_L b^\prime) \vee_L b = (a \vee_L b) \vee_L b =  
a \vee_L b = f(a \vee_B b) = f(s). 
$$  
Let $y\in (\B^s)_\alpha$. There is a corresponding element $y^\prime =
y \vee_B\alpha$ in $(\B^s)^\alpha$. We may write $ y= x\vee_B s$ for
some $x\in \B_s$ and since $s\vee_B \alpha = s\vee_B s\vee_B s^\prime
= s\vee_B s^\prime = \alpha$ we have  
$$ 
y^\prime = x \vee_B s \vee_B \alpha = x \vee_B \alpha.
$$
Applying $f$, while recaling that $f\alpha=fs$, gives
$$
f(y^\prime)  = f (x \vee_B \alpha) = fx \vee_L f\alpha =  fx \vee_L fs = f(x \vee_B s) = f(y).
$$
It follows from the definition of the induced sheaf on the Boolean
cover that the map $F^{y^\prime}_y$ is the identity. This shows that
$\B^s$ is a double with respect to the atom $\alpha$.  

To finish the proof of (\ref{eq:7}), we use Proposition
\ref{section3:prop:doubles} and the long exact sequence resulting from   
$$
0\ra
C_*(\B_s;F)\ra
C_*(\B;F)\ra
C_*(\B^s;F)\ra
0.
$$
We may now repeat this process by taking a sequence of deletions of
$\B$ until we arrive at $\widetilde{L^a}$. Courtesy of  
(\ref{eq:7}), the homology remains unchanged at each step, giving the required result. \qed  
\end{proof}

The previous result allows us to relate the cellular homology of the
Boolean cover of a lattice with the homology of the Boolean covers of
the restriction and deletion. 

\begin{theorem}[Deletion-Restriction Long Exact Sequence]
\label{section3:theorem:deletion.restriction}
Let $L$ be a geometric lattice equipped with a sheaf $F$ and let $f:\widetilde{L}\ra
L$ be its Boolean cover. Then for any atom $a\in L$ there is a long exact sequence
$$
\cdots
\ra
\HC{i}(\widetilde{L^a};F)
\ra
\HC{i}(\widetilde{L_a};F)
\ra
\HC{i}(\widetilde{L};F)
\ra
\HC{i-1}(\widetilde{L^a};F)
\ra
\HC{i-1}(\widetilde{L_a};F)
\ra
\cdots
$$
\end{theorem}

\begin{proof}
If $A$ are the atoms of $L$ then we can use the short exact sequence
(\ref{eq:6}), induced by a sub-Boolean with 
$x=A\setminus\{a\}$, to get a short exact sequence
$$
0\ra C_*(B_a;F)\ra C_*(\widetilde{L};F)
\ra C_{*-1}(B^a;F)\ra 0  
$$
where, as above, $B=\widetilde{L}$. 
The result follows by applying Theorem
\ref{section3:theorem:deletion.restriction.auxilary} and (\ref{eq:Bsuba}) to the resulting
long exact sequence. 
 \qed
\end{proof}


\section{Sheaves on hyperplane arrangements}

We return to the arrangement lattices of \S\ref{section1:posets},
the natural sheaf and its exterior powers from
\S\ref{section1:sheaves}.  In \S\ref{section4:gradedeuler} we discuss
graded Euler characteristics and their computation in terms of
characteristic polynomials. \S\ref{section4:exterior} gives a complete
calculation of the homology of an essential hyperplane arrangement
with coefficients in the exterior natural sheaf.
\S\ref{section4:exterior2} extends this to the non-essential case.   

Throughout, $V$ is a finite dimensional vector
space over a field $k$ (initially arbitrary, then
restricted to a subfield of $\C$ in \S\ref{section4:exterior2}); $A$
is an arrangement 
in $V$ and $L=L(A)$ is the intersection lattice; $F$ is the
natural sheaf associated to $A$.


\subsection{Graded Euler characteristics}
\label{section4:gradedeuler}

For $F$ the natural sheaf on the arrangement lattice $L$ we have the
exterior sheaf $\Lambda^{\blob} F$. The graded Euler characteristic of
the cellular homology 
of $\widetilde{L}$ with coefficients in $\Lambda^{\blob} F$ turns out to
be very close to the characteristic polynomial of the arrangement
 lattice $L$.

\begin{proposition} 
\label{exterior:eulerchar:cellular}
The graded Euler chracteristic 
${\displaystyle \chi_q\HC * (\widetilde{L};\Lambda^\blob F) = \chi_L(1+q)}.$
\end{proposition}

\begin{proof}
This is a straight-forward calculation (recalling that $\dim F(x) = \dim x$):
\begin{eqnarray*}
\chi_q\HC * (\widetilde{L};\Lambda^\blob F) & = & \sum_k q^k  \chi
                                                  \HC *
                                                  (\widetilde{L};\Lambda^k
                                                    F)
  =   \sum_k q^k \sum_{x\in L} \mu_L(\0,x) \dim \Lambda^k F(x)
       \\ 
  & = &  \sum_k q^k \sum_{x\in L} \mu_L(\0,x) {\dim x \choose k}
    =   \sum_{x\in L} \mu_L(\0,x)  \sum_k q^k {\dim x \choose k} \\
   & = &  \sum_{x\in L} \mu_L(\0,x) (1+q)^{\dim x} = \chi_L(1+q)\\
\end{eqnarray*}
where we have used Corollary \ref{cor:coverchar} at the second equality.
\qed
\end{proof}

From this, another application of Corollary \ref{cor:coverchar}
gives the graded Euler characteristic for the sheaf homology:

\begin{corollary}
\label{exterior:eulerchar:sheaf}
\hspace{1em}
${\displaystyle \chi_q\HS * (L\kern-1pt\setminus \kern-1pt
  \0;\Lambda^\blob F) =
- \chi_L(1+q)+(1+q)^{\dim V}}.$
\end{corollary}

\subsection{The exterior sheaf on an essential arrangement}
\label{section4:exterior}

We focus first on {\em essential} arrangements -- those for which the
intersection of all hyperplanes is trivial.

\begin{theorem}
\label{exterior:cellular:essential}
Let $L$ be the intersection lattice of an essential hyperplane
arrangement in a space $V$, let
$F$ be the natural sheaf on $L$ and $\Lambda^{j}F$ be the
$j$-th exterior power of $F$.
If $\rk(L)\geq 2$ and $\widetilde{L}\ra L$ is the Boolean cover of $L$, then
$\HC{i}(\widetilde{L};\Lambda^{j}F)$ is trivial unless $0\leq i<\rk(L)$
and 
$i+j=\rk(L)=\dim V$, in which case:
$$
\dim\HC{i}(\widetilde{L};\Lambda^{j}F)=
\frac{{(-1)^i}}{j!}
\chiprimek j
$$
where  $\chi^{({j})}_{L}(t)$ is the $j$-th derivative of the characteristic polynomial of $L$.
\end{theorem}

\begin{figure}
\begin{tikzpicture}
\draw [white] (0,0)--(\textwidth,0); 
%
\node at (7,3.5)
{\includegraphics[scale=0.35]{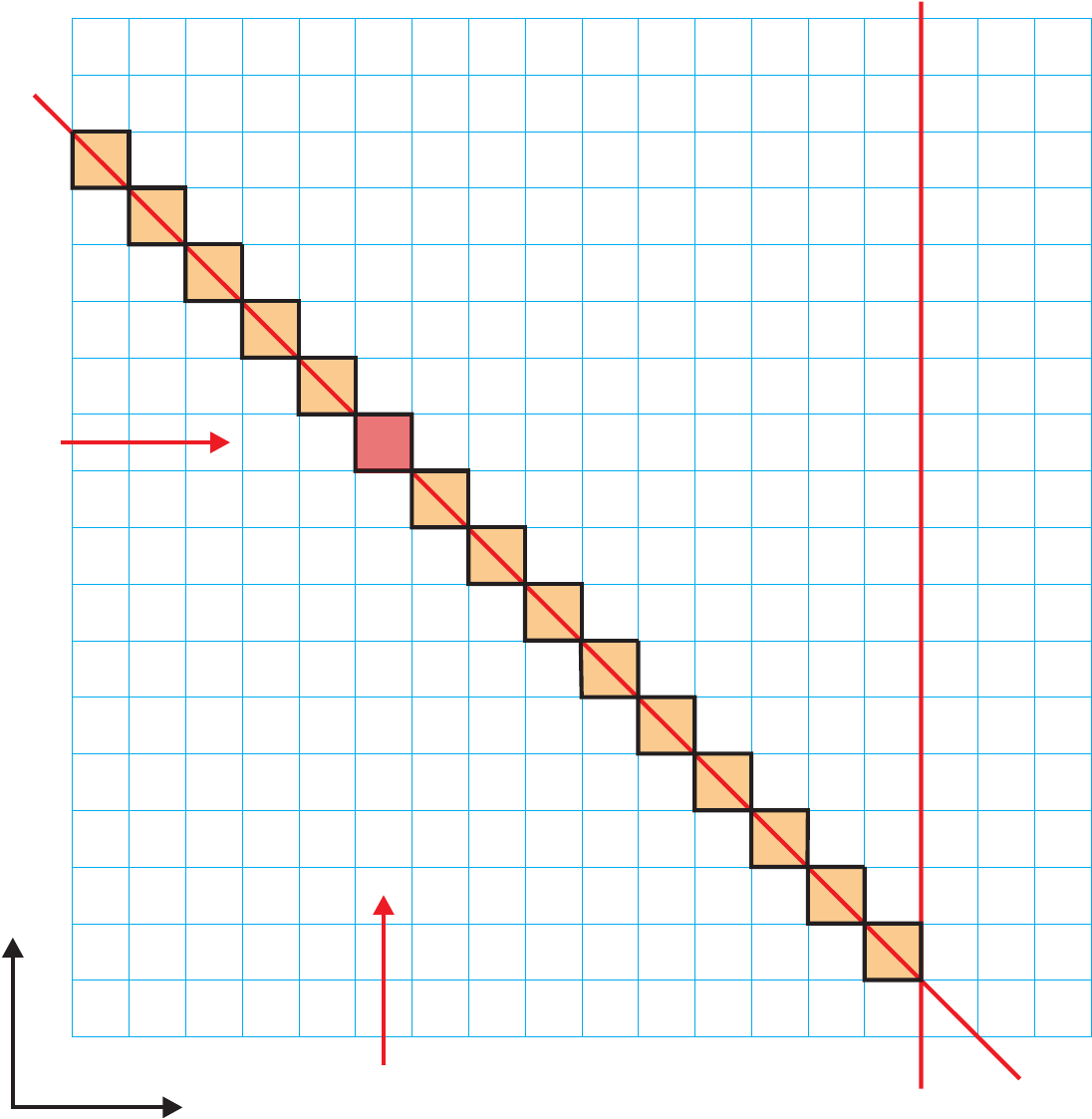}};
\node[color=red] at (3.8,4.2){$\scriptstyle{j}$};
\node[color=red] at (6,0.25){$\scriptstyle{i}$};
\node at (3.85,1.4){$\scriptstyle{j}$};
\node at (4.95,0.25){$\scriptstyle{i}$};
\node[color=red] at (9.6,3.5){$\scriptstyle{i\geq\rk(L)}$};
\node[color=red] at (10.5,0.65){$\scriptstyle{i+j=\rk(L)=\dim V}$};
%
%
\end{tikzpicture}
\caption{Support of $\HC i (\widetilde{L};\Lambda^jF)$ for $L$
essential in the
$(i,j)$-plane, with the dimension of the red square given in Theorem
\ref{exterior:cellular:essential}.}
\label{fig:cellular:exterior:support2}
\end{figure}
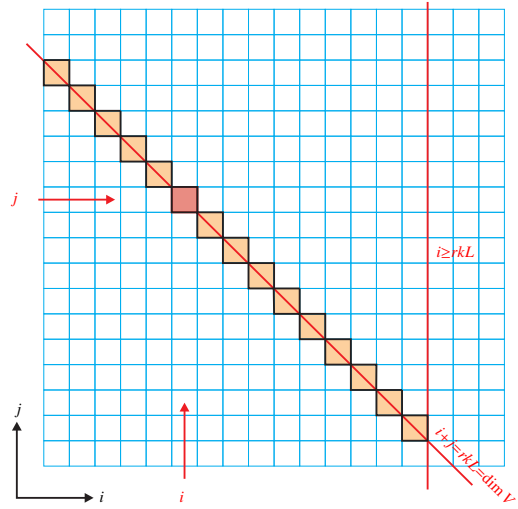

The support of $\HC{i}(\widetilde{L};\Lambda^{j}F)$ is shown in Figure 
\ref{fig:cellular:exterior:support2}.
The remainder of
the section is devoted to the proof of Theorem 
\ref{exterior:cellular:essential}, which is broken down into several subparts:
 
\paragraph{The proof of Theorem \ref{exterior:cellular:essential}
  when $L$ itself 
  is Boolean.} We prove the result where $L$ itself is Boolean (and hence
$\widetilde{L}=L$) separately from the general case. The
characteristic polynomial is given by
$$
\chi_L(t)=(t-1)^{\rk(L)}
$$
and we require:

\begin{proposition}
\label{exterior:cellular:booleans}
Let $L$ be a Boolean intersection lattice of an essential hyperplane
arrangement in a space $V$, 
let $F$ be the natural sheaf on $L$ and $\Lambda^{j}F$ be the
$j$-th exterior power of $F$. If $\rk(L)\geq 1$ then
$$
\dim \HC{i}(L;\Lambda^{j}F)
= \begin{cases} 
1 & \text{ if $i=0$ and $j=\rk(L)$ }\\
0 & \text{ else.}
\end{cases}
$$
\end{proposition}

\begin{proof}
The proof is an induction on the rank of $L$. The case where $\rk(L)=1$
(and so we have a single hyperplane the 
trivial space in a $1$-dimensional $V$) can be verified by brute
force. Otherwise, for $\rk(L)\geq 2$,
there is a basis $Z=\{v_1,\ldots,v_n\}$ for $V$ such that $L$ is the
lattice of subsets of $Z$ ordered via reverse inclusion, with the
subset corresponding to $x\in L$ giving a basis for $F(x)$. 
If $a\in A$ is a hyperplane with basis 
$\{v_1,\ldots,\hat{v}_i,\ldots,v_n\}$, then for $x\in L_a$
the space $F(x\vee a)$ has basis some subset
$\{u_1,\ldots,u_m\}$ of $Z\setminus \{v_i\}$ and $F(x)$ has basis
$\{v_i,u_1,\ldots,u_m\}$.

A basis vector $u_{i_1}\wedge\cdots\wedge u_{i_j}$ of $\Lambda^jF(x)$
may or may not contain the element $v_i$, leading to a decomposition 
\begin{equation}\label{eq:lambda}
\Lambda^jF(x) \cong \Lambda^{j-1}F(x\vee a) \oplus \Lambda^jF(x\vee a).
\end{equation}
Writing $G=\Lambda^j F$ for $j\geq 1$ (and similarly $G_a$ and $G^a$) the structure maps
$$G_{x}^{x\vee a} \colon G(x\vee a) = \Lambda^j F (x\vee a) \ra
\Lambda^j F(x) = G(x)$$ are the obvious inclusions, so $G$ is
decomposable  in the sense of 
\S\ref{section3:splittingbooleans}. Moreover, the isomorphism
(\ref{eq:lambda}) leads to an isomorphism of sheaves 
\begin{equation}
  \label{eq:12}
G_a/G^a\cong \Lambda^{j-1}F^a    
\end{equation}
(some care is needed in checking what happens when $j$ is close to
$\dim x$, as some of the spaces become $0$). Theorem
\ref{section3:theorem:decomposing:booleans} thus gives
\begin{equation}
  \label{eq:10}
\HC{*}(L;\Lambda^jF)\cong\HC{*}(L_a;\Lambda^{j-1}F^a)
\end{equation}
with $F^a$ the natural sheaf of the essential arrangement $A^a$ on
$L_a\cong L^a$ (both are Boolean of one smaller rank than $L$). Hence, by induction
$$
\dim \HC{i}(L;\Lambda^jF) = \dim \HC{i}(L_a;\Lambda^{j-1}F^a) = 
\begin{cases}
1 & \text{ if $i=0$ and $j-1=\rk(L) -1$ }\\
0 & \text{ else.}
\end{cases}
$$
\qed
\end{proof}

Suppose now that $L$ is not Boolean. We argue by induction on the
number $|A|$ of hyperplanes. Throughout, if $j>\rk(L)=\dim V$ then the
sheaf $\Lambda^j F=0$; we thus need only consider  $j$ in the range
$0\leq j\leq \rk(L)=\dim V$.
When $|A|=1$ or $2$, the intersection
lattice $L(A)$ is Boolean of rank $|A|$, and so these cases have been
handled already. 

\paragraph{The base case $|A|=3$.} We saw in
\S\ref{section1:posets} that the only non-Boolean $L$ on three
hyperplanes is realised by a braid arrangement, and with $L$ 
isomorphic to the partition lattice $\Pi(3)$. When essential, the
arrangement lives in a $2$-dimensional $V$ with basis $\{v_1,v_2\}$
and consists of the lines spanned by $v_1,v_2$ and $-v_1-v_2$. 

The characteristic polynomial (see \cite{Stanley12}*{\S 3.10.4}) is
$
\chi_{L}(t)=(t-1)(t-2)
$
and Theorem \ref{exterior:cellular:essential} states:\\
\hspace*{1cm} $j=0$: $\dim\HC{i}(\widetilde{L};\Lambda^{0}F)= 0$ for all $i$, \\
\hspace*{1cm} $j=1$: $\dim\HC{i}(\widetilde{L};\Lambda^{1}F)=
\begin{cases}
1 & \text{ if $i=1$}\\
0 & \text{ else,}
\end{cases} $\\
\hspace*{1cm} $j=2$: $\dim\HC{i}(\widetilde{L};\Lambda^{2}F)=
\begin{cases}
1 & \text{ if $i=0$}\\
0 & \text{ else.}
\end{cases} $\\
To prove this, each case is treated separately. For $j=0$ the sheaf is
constant $\Lambda^0F=\Delta k$ and so the induced sheaf on the Boolean
cover is also constant and by applying Theorem
\ref{section3:theorem:decomposing:booleans} we have 
$$
\HC{i}(\widetilde{L};\Delta k)=
\HC{i}(\widetilde{L}_a;\text{zero sheaf})=0,
\text{ for all }i.
$$
 For $j=1$, we have the natural sheaf and the induced sheaf 
on the Boolean cover $\widetilde{L}$ has constant value $F(\1)$ on all
the elements of ranks $2$ and $3$.
Two applications of Theorem
\ref{section3:theorem:decomposing:booleans} give the required
result. For $j=2$, the sheaf 
$\Lambda^2F$ is trivial except at $\0\in L$ where it is
$1$-dimensional with basis $v_1\wedge v_2$, once again 
in happy agreement with Theorem \ref{exterior:cellular:essential}. 

\paragraph{The vanishing degrees in the general case $|A|>3$.} 
We may assume that $L$ is non-Boolean and so, by the results of
\S\ref{section1:posets}, 
it has a dependent atom $a\in A$. 

The deletion $L_a$ is then an essential arrangement lattice having $|A|-1$
hyperplanes and $\rk(L_a)=\rk(L)$, by the dependence of $a$. The sheaf
$F_a$, which is just $F$ restricted to $L_a$, is the natural sheaf
of this arrangement, and the restriction of $\Lambda^j F$ to
$L_a$ is just $\Lambda^j F_a$. 
The
restriction $L^a$ is an essential arrangement lattice having at most $|A|-1$
hyperplanes and $\rk(L^a)=\rk(L)-1$. The sheaf
$F^a$ is the natural sheaf
of this arrangement and $\Lambda^j F^a$ is the restriction of
$\Lambda^j F$ to $L^a$.

Both $L_a$ and $L^a$ are
either Boolean, or essential arrangement lattices on fewer than $|A|$
hyperplanes, hence 
come under the auspices of the inductive
hypothesis. The deletion-restriction long exact sequence,
Theorem \ref{section3:theorem:deletion.restriction}, gives
$$
\cdots
\ra
\HC{i}(\widetilde{L_a};\Lambda^j F_a)
\ra
\HC{i}(\widetilde{L};\Lambda^j F)
\ra
\HC{i-1}(\widetilde{L}^a;\Lambda^j F^a)
\ra
\cdots
$$
If $i\not=\rk(L)-j$ or $i=\rk(L)$,  then both the left and right terms
vanish, hence by induction we 
get $\HC{i}(\widetilde{L};\Lambda^j F)=0$ as required.

\paragraph{The non-vanishing degree in the general case $|A|>3$.}

For fixed $j$ we have shown that there is only one non-trivial group
among the $\HC{i}(\widetilde{L};\Lambda^j F)$, namely when $i=
\rk(L)- j$. This reduces the task to an Euler characteristic
computation. We have 
$$
\chi \HC{*}(\widetilde{L};\Lambda^j F) = \sum_n (-1)^n \dim
\HC{n}(\widetilde{L};\Lambda^j F) = (-1)^{\rk(L)-j}\dim
\HC{\rk(L)-j}(\widetilde{L};\Lambda^j F). 
$$
From this we get
$$
\chi_q\HC * (\widetilde{L};\Lambda^\blob F)  =  \sum_j q^j  \chi\HC *
(\widetilde{L};\Lambda^j F) 
  =  \sum_j q^j  (-1)^{\rk(L)-j}\dim \HC{\rk(L)-j}(\widetilde{L};\Lambda^j F).
$$
Thus, for $i= \rk(L)- j$,
\begin{eqnarray*}
\dim \HC{i}(\widetilde{L};\Lambda^j F) &= &(-1)^i \times
                                            (\text{coefficient of
                                            $q^j$ in $\chi_q\HC *
                                            (\widetilde{L};\Lambda^\blob
                                            F)$})\\ 
& = & (-1)^i \times (\text{coefficient of $q^j$ in $\chi_L(1+q)$})
\end{eqnarray*}
where the last equality is due to Proposition
\ref{exterior:eulerchar:cellular}. The Taylor expansion of (the
polynomial) $\chi_L(1+q)$ immediately reveals the coefficient of $q^j$
in $\chi_L(1+q)$ to be $\frac{1}{j!}\;\chiprimek j$ giving 
$$
\dim \HC{i}(\widetilde{L};\Lambda^j F) = \frac{(-1)^i}{j!}\chiprimek j .
$$


This completes the proof of Theorem \ref{exterior:cellular:essential}
computing the cellular homology. Our real interest is in the sheaf
homology, which we can now compute by applying Proposition
\ref{section3:fiddlingwith0:sheaf} to 
Theorem \ref{exterior:cellular:essential}.

\begin{theorem}
\label{exterior:sheaf-essential}
Let $L$ be the intersection lattice of an essential hyperplane
arrangement in a space $V$. Let
$F$ be the natural sheaf on $L$ and $\Lambda^{j}F$ be the
$j$-th exterior power of $F$.
If $\rk(L)\geq 2$ then 
$\HS i (L\kern-1pt\setminus\kern-1pt\0;\Lambda^{j}F)$ is
trivial unless: \\ 
\hspace*{1cm}-- either
$0<i<\rk(L)-1$ and $i+ j = \rk(L) - 1$,
in which case
$$ 
\dim\HS {i}
(L\kern-1pt\setminus\kern-1pt\0;\Lambda^{j}F) =
\frac{ (-1)^{i+1}}{j!} \chiprimek j
$$
\hspace*{1cm} -- or, $i=0$ and  or $j= \rk(L) -1$, in which case
$$
\dim\HS {0}
(L\kern-1pt\setminus\kern-1pt\0;\Lambda^{j}F) =
 \binom{\rk(L)}{j}
- 
\frac{1}{j!}
\chiprimek j
 $$
\hspace*{1cm} -- or, $i=0$ and $j<\rk(L) -1$, in which case 
$$
\dim\HS {0}
(L\kern-1pt\setminus\kern-1pt\0;\Lambda^{j}F) =
\binom{\rk(L)}{j}
 $$
 where  $\chi^{({j})}_{L}(t)$ is the $j$-th derivative of the
 characteristic polynomial of $L$. 
\end{theorem}


\subsection{The exterior sheaf for a non-essential arrangement}
\label{section4:exterior2}
In this section we assume our field to be a sub-field of $\C$ and so
we may assume that the vector space $V$ comes equipped with an inner
product $\langle - , - \rangle$. Let $L$ be the intersection lattice
of a hyperplane 
arrangement in a space $V$ with $U=\bigcap_{a\in A} a$. Let
$F$ be the natural sheaf on $L$.

For a subspace $B\subset V$ such that $U\subset B \subset V$ we define
the orthogonal complement of $U$ in $B$ to be 
$$
U^{\perp B} = \{ b\in B \mid \langle b,u\rangle = 0,\text{ for all }u\in U\}.
$$
 Note that $U^{\perp B}$ is a subspace of $B$ and 
 $B= U \oplus U^{\perp B}$. (This last condition may fail in finite
 characteristic.) Moreover, if $U\subset B\subset B^\prime \subset V$,
 there is an inclusion  $ \iota\colon U^{\perp B} \subset U^{\perp
   B^\prime}$ and with respect to the decompositions $B= U \oplus
 U^{\perp B}$ and $B^\prime= U \oplus U^{\perp B^\prime}$ the
 inclusion $B\subset B^\prime$ decomposes as $1\oplus \iota$.

We define a sheaf $F^\perp$ on $L$ as follows. For $x\in L$ we set
$F^\perp(x) = U^{\perp F(x)}$. If $x\leq y$ then the structure map
$F^\perp(x\leq y)$ is the inclusion $ F^\perp(y) = U^{\perp F(y)}
\subset U^{\perp F(x)} = F^\perp(x)$ induced by the inclusion  $F(y)
\subset F(x)$.  

We have:
 
\begin{lemma} There is a direct sum decomposition of sheaves $ F =
  \Delta U \oplus F^\perp$.  
\end{lemma}

By definition $F^\perp$ is a sub-sheaf of $F$ on $L$, but it is also
the natural sheaf of an essential hyperplane arrangement in $U^{\perp
  V}$. There is one hyperplane $U^{\perp H}$ for each hyperplane $H$
of the original arrangement and courtesy of the relation 
$
U^{\perp (B\cap C)} = U^{\perp B} \cap U^{\perp C}
$
we see that the lattice of this new arrangement is again $L$. 
The natural sheaf on the new arrangement is precisely $F^\perp$, seen
immediately from the definition of $F^\perp$. We refer to $F^\perp$ as
the {\em essentialisation} of $F$. 

We will continue to write $\chi_L(t)$ for the characteristic
polynomial of $L$ equipped with the natural sheaf $F$, that is to say, 
$ \chi_L(t) = \sum \mu_L(\0, x)t^{\dim x}$. Writing
$\chi_{(L,F^\perp)}(t)$ for the characteristic polynomial of the
essentialisation we easily see 
$$
\chi_L(t) = t^{\dim U} \chi_{(L,F^\perp)}(t).
$$

 Let $\Lambda^{j}F$ be the
$j$-th exterior power of $F$ (the natural sheaf).
Theorem \ref{exterior:cellular-non-essential} below gives the cellular homology
$\HC{i}(\widetilde{L};\Lambda^{j}F)$
and the support in the $(i,j)$-plane is illustrated in Figure
\ref{fig:cellular:exterior:support}. 

\begin{theorem}
\label{exterior:cellular-non-essential}
Let $L$ be the intersection lattice of a hyperplane
arrangement in a space $V$ with $U=\bigcap_{a\in A} a$. Let
$F$ be the natural sheaf on $L$ and $\Lambda^{j}F$ be the
$j$-th exterior power of $F$.
If $\rk(L)\geq 2$ and $\widetilde{L}\ra L$ is the Boolean cover of $L$, then
$\HC{i}(\widetilde{L};\Lambda^{j}F)$ is trivial unless 
$0\leq i<\rk(L)$ and $\rk(L)\leq i+j\leq \dim V$, in which case:
$$
\dim\HC{i}(\widetilde{L};\Lambda^{j}F)=
\frac{ (-1)^i}{({\rk(L) - i})!}
\binom{\dim U}{i+j-\rk(L)}\;
\chiprimeperpk {\rk(L) - i} .
$$
with $\chi^{({k})}_{(L,F^\perp)}(t)$ the $k$-th
derivative of the characteristic polynomial of the essentialisation of
$L$. 
\end{theorem}

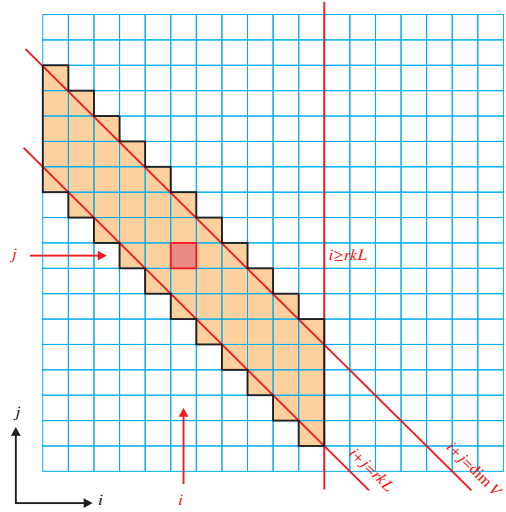
\begin{figure}
\begin{tikzpicture}
\draw [white] (0,0)--(\textwidth,0); 
%
\node at (7,3.5)
{\includegraphics[scale=0.35]{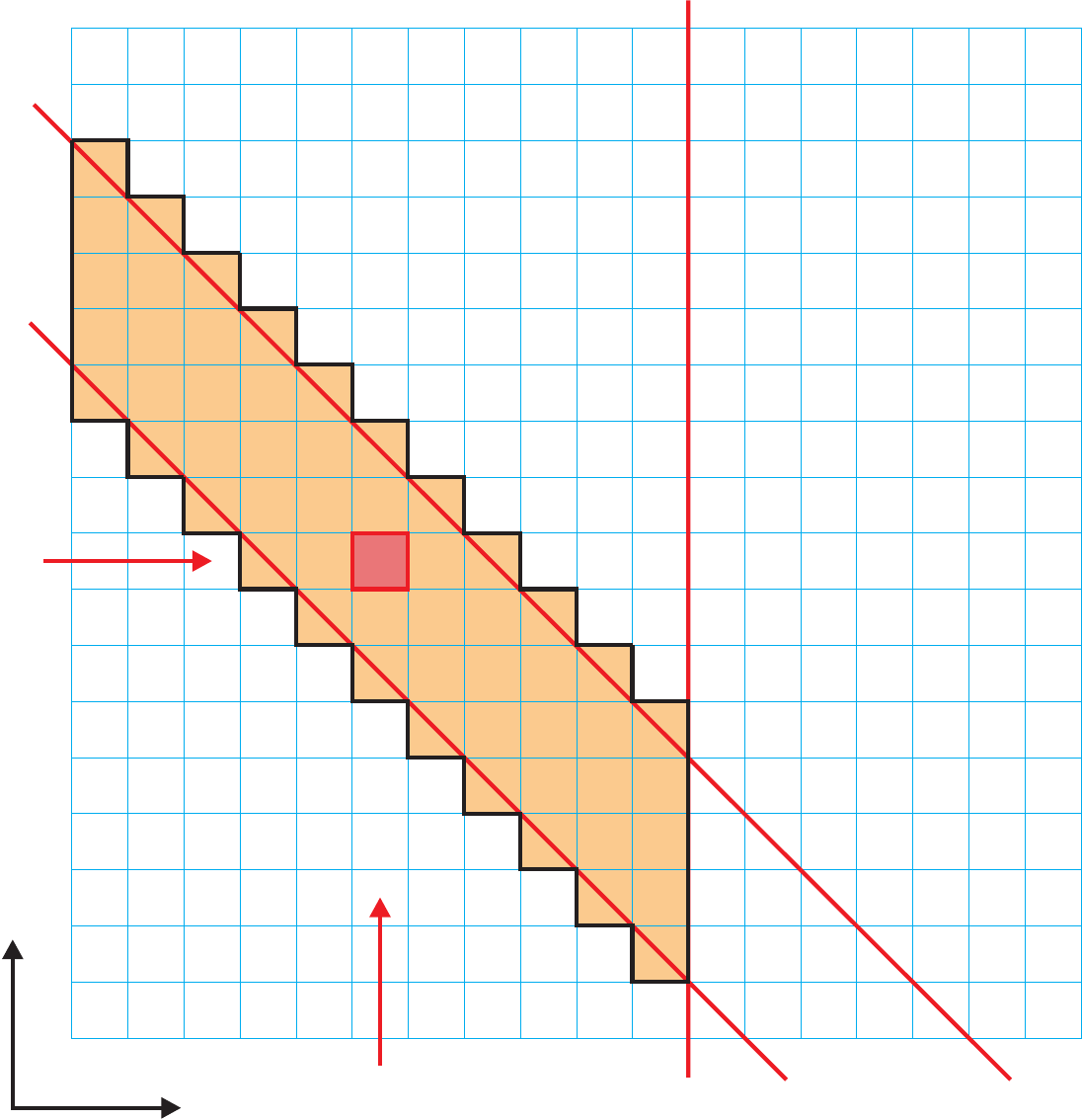}};
\node[color=red] at (3.8,3.5){$\scriptstyle{j}$};
\node[color=red] at (6,0.25){$\scriptstyle{i}$};
\node at (3.85,1.4){$\scriptstyle{j}$};
\node at (4.95,0.25){$\scriptstyle{i}$};
\node[color=red] at (8.25,3.5){$\scriptstyle{i\geq\rk(L)}$};
\node[color=red] at (10.15,0.65){$\scriptstyle{i+j=\dim V}$};
\node[color=red] at (8.75,0.65){$\scriptstyle{i+j=\rk(L)}$};
%
%
\end{tikzpicture}
\caption{Support of $\HC i (\widetilde{L};\Lambda^jF)$ in the
$(i,j)$-plane, with the dimension of the red square given by Theorem
\ref{exterior:cellular-non-essential}.}
\label{fig:cellular:exterior:support}
\end{figure}

\begin{proof}
The decomposition $F = \Delta U \oplus F^\perp$ allows us to write
$$
\Lambda^j F = \Lambda^j (\Delta U \oplus F^\perp) = \bigoplus_{s+t=j}
\Lambda^s \Delta U \otimes \Lambda^tF^\perp 
=  \bigoplus_{s+t=j} \Delta \Lambda^s U \otimes \Lambda^tF^\perp.
$$
Applying, Corollary \ref{cor:directsum} gives
$$
\HC i (\widetilde{L}; \Lambda^j F) =  \bigoplus_{s+t=j} \HC i
(\widetilde{L}; \Delta \Lambda^s U \otimes \Lambda^tF^\perp). 
$$
Recalling that we are working in characteristic zero, the universal
coefficient theorem tells us that 
$$
 \HC * (\widetilde{L}; \Delta \Lambda^s U \otimes \Lambda^tF^\perp)
 \cong   \Lambda^s U \otimes \HC * (\widetilde{L};  \Lambda^tF^\perp)  
$$
and we have
\begin{equation}
\label{eq:dimdecomp}
\dim\HC i (\widetilde{L}; \Delta \Lambda^s U \otimes \Lambda^tF^\perp) = 
\binom{\dim U}{s}\dim \HC i (\widetilde{L};  \Lambda^tF^\perp).
\end{equation}
The dimension on the right can be computed from Theorem
\ref{exterior:cellular:essential} because $F^\perp$ is essential:  
we have that $ \dim  \HC i (\widetilde{L};  \Lambda^tF^\perp) =0$
unless $t=\rk(L) -i$, in which case 
$$ 
\dim (\HC i (\widetilde{L};  \Lambda^{\rk(L) -i}F^\perp) =
\frac{(-1)^i}{(\rk(L) -i)!} \chiprimeperpk {\rk(L) - i}. 
$$ 
Since $s+t=j$ we have $s= i+j-\rk(L)$ which means given $i,j$ the
values of $s$ and $t$ must be taken to be $t=\rk(L) -i$ and $s= i+j-\rk(L)$  and so we get 
$$
\dim \HC i (\widetilde{L}; \Lambda^j F) = \frac{ (-1)^i}{({\rk(L) - i})!}
\binom{\dim U}{i+j-\rk(L)}\;
\chiprimeperpk {\rk(L) - i} .
$$
What remains is to find the values of $i$ and $j$ for which this
computation is valid. The conditions are (i) $0\leq i < \rk(L)$ (in
order to be able to apply Theorem \ref{exterior:cellular:essential}),
(ii) $s\leq \dim U$ (otherwise $\Lambda^s U$ is trivial) and (iii)
$t\leq j$ (since $s+t = j$ in the sum above). Condition (i) is seen in
the statement of the theorem; since $\dim U = \dim V - \rk(L)$,
condition (ii) implies $i+j\leq \dim V$; condition (iii) implies
$\rk(L)\leq i+j$.
So conditions  (ii) and (iii) together give the other
condition in the statement of the theorem, namely $ \rk(L) \leq i+j
\leq \dim V$. \qed 
\end{proof}

As in the essential case, we can convert this into a result about 
sheaf homology. As before we write $\chi^{({k})}_{(L,F^\perp)}(t)$ for
the $k$-th derivative of the characteristic polynomial of the
essentialisation of $L$. 

\begin{theorem}
\label{exterior:sheaf-non-essential}
Let $L$ be the intersection lattice of a hyperplane
arrangement in a space $V$ with $U=\bigcap_{a\in A} a$. Let
$F$ be the natural sheaf on $L$ and $\Lambda^{j}F$ be the
$j$-th exterior power of $F$.
If $\rk(L)\geq 2$ then 
$\HS i (L\kern-1pt\setminus\kern-1pt\0;\Lambda^{j}F)$ is
trivial unless: \\ 
\hspace*{1cm}-- either
$0<i<\rk(L)-1$ and $\rk(L)\leq i+j+1\leq \dim V$,
in which case
$$ 
\dim\HS {i}
(L\kern-1pt\setminus\kern-1pt\0;\Lambda^{j}F) =
\frac{ (-1)^{i+1}}{({\rk(L) - i - 1})!}
\binom{\dim U}{i+1+j-\rk(L)}\;
\chiprimeperpk {\rk(L) - i - 1}
$$
\hspace*{1cm}-- or, $i=0$ and $\rk(L)\leq j< \dim V$, in which case 
$\dim\HS {0}
(L\kern-1pt\setminus\kern-1pt\0;\Lambda^{j}F)$ equals
$$
\binom{\dim V}{j}
- 
\frac{1}{{\rk(L)}!}
\binom{\dim U}{j-\rk(L)}\;
\chiprimeperpk {\rk(L)}
   -
  \frac{1}{({\rk(L)-1})!}
\binom{\dim U}{j+1-\rk(L)}\;
\chiprimeperpk {\rk(L) -1}
 $$
\hspace*{1cm}-- or, $i=0$ and $j= \rk(L) -1$, in which case 
$$
\dim\HS {0}
(L\kern-1pt\setminus\kern-1pt\0;\Lambda^{\rk(L) -1}F) =
\binom{\dim V}{\rk(L)-1}
   -
  \frac{1}{({\rk(L)-1})!}\;
  \chiprimek {\rk(L) -1}
$$
\hspace*{1cm}-- or, $i=0$ and $j< \rk(L) -1$, in which case 
$$
\dim\HS {0}
(L\kern-1pt\setminus\kern-1pt\0;\Lambda^{j}F) =
\binom{\dim V}{j}.
$$
\end{theorem}

%
%
%

%
%
%

\end{document}